\documentclass[11pt]{amsart}
\usepackage[utf8]{inputenc}

\usepackage[nospace,noadjust]{cite}
\usepackage{relsize}
\usepackage{cite}
\usepackage{color}
\usepackage[dvipsnames]{xcolor}

\usepackage{tikz-cd}
\usetikzlibrary{calc,fit,matrix,arrows,automata,positioning}
\usetikzlibrary{decorations.pathreplacing,angles,quotes}
\usepackage{colortbl}
\usepackage{hyperref, enumitem}
\hypersetup{
  colorlinks   = true, 
  urlcolor     = blue, 
  linkcolor    = Purple, 
  citecolor   = red 
}
\usepackage{amsmath,amsthm,amssymb}

\usepackage[margin=2.5cm]{geometry}
\usepackage{enumerate}
\usepackage{amsmath}
\usepackage{amssymb,latexsym}
\usepackage{amsthm}
\usepackage{color}
\usepackage{cancel}
\usepackage{graphicx}
\usepackage{cite}
\usepackage{changes}
\usepackage{lineno}
\usepackage{hyperref}
\usepackage{comment,xcolor}
\usepackage{amscd}
\usepackage[arrow]{xy}

\DeclareMathOperator{\C}{\mathcal{C}}

\DeclareMathOperator{\supp}{supp}
\DeclareMathOperator{\rk}{rk}

\newtheorem{theorem}{Theorem}[section]

\newtheorem{lemma}[theorem]{Lemma}
\newtheorem{corollary}[theorem]{Corollary}
\newtheorem{definition}[theorem]{Definition}
\newtheorem{proposition}[theorem]{Proposition}

\newtheorem{example}[theorem]{Example}

\newtheorem{remark}[theorem]{Remark}

\newtheorem{construction}[theorem]{Construction}

\newcommand{\fqm}{\mathbb{F}_{q^m}}

\newcommand{\cC}{{\mathcal C}}

\newcommand{\cM}{{\mathcal M}}

\newcommand{\F}{{\mathbb F}}

\newcommand{\GL}{\hbox{{\rm GL}}}

\newcommand{\fq}{{\mathbb F}_{q}}

\newcommand{\la}{\langle}
\newcommand{\ra}{\rangle}

\newcommand{\PG}{\mathrm{PG}}
\newcommand{\N}{\mathrm{N}}

\DeclareMathOperator{\Ext}{Ext}

\DeclareMathOperator{\HH}{\mathrm{H}}

\newcommand{\bfn}{\mathbf{n}}
\newcommand{\Fmnkd}{[\bfn,k,d]_{q^m/q}}
\newcommand{\Fm}{\mathbb{F}_{q^m}}

\newcommand{\Fmnk}{[\bfn,k]_{q^m/q}}

\title{Geometric dual and sum-rank minimal codes}
\usepackage[foot]{amsaddr}
\author{Martino Borello$^1$}
\author{Ferdinando Zullo$^2$}
\address{$^1$Universit\'e Paris 8, Laboratoire de G\'eom\'etrie, Analyse et Applications, LAGA, Universit\'e Sorbonne Paris Nord, CNRS, UMR 7539, France.}
\address{$^2$Dipartimento di Matematica e Fisica, Universit\`a degli Studi della Campania ``Luigi Vanvitelli'', I--\,81100 Caserta, Italy}

\email{martino.borello@univ-paris8.fr, ferdinando.zullo@unicampania.it}

\begin{document}
\maketitle

\begin{abstract}
The main purpose of this paper is to further study the structure, parameters and constructions of the recently introduced minimal codes in the sum-rank metric. These objects form a bridge between the classical minimal codes in the Hamming metric, the subject of intense research over the past three decades partly because of their cryptographic properties, and the more recent rank-metric minimal codes. We prove some bounds on their parameters, existence results, and, via a tool that we name geometric dual, we manage to construct minimal codes with few weights. A generalization of the celebrated Ashikhmin-Barg condition is proved and used to ensure minimality of certain constructions.\\


\end{abstract}

\tableofcontents

\section*{Introduction}

Sum-rank metric constitutes a bridge between the more classical Hamming and rank metrics, which may be considered as its special cases. It has been used implicitely first in \cite{el2003design,lu2005unified} and explicitely introduced in the context of network coding in \cite{nobrega2010multishot}. One of the main reasons why this metric has been in the spotlight in recent years is the fact that sum-rank metric codes outperform the more classical ones in
terms of the required field size to construct codes achieving the Singleton bound in the corresponding metric \cite{Martinez2018skew}. This is due to the existence of the so-called linearized Reed–Solomon codes, a family of maximum sum-rank distance (MSRD) codes with polynomial field sizes. 
In the survey paper \cite{martinez2022codes}, the interested reader may find a very
detailed summary on properties and applications of sum-rank
metric codes in distributed storage systems, network coding, and multi-antenna communication.

The main purpose of this paper is to further study the structure, parameters and constructions of minimal codes in the sum-rank metric, recently introduced in \cite{santonastaso2022subspace}. Minimal codes are classical objects in the Hamming metric rich of connections with different areas of mathematics, such as cryptography \cite{massey1993minimal}, finite geometry \cite{Alfarano2020ageometric,tang2021full}, and combinatorics \cite{bishnoi2023blocking}. One of the main concerns about these objects is to find bounds on their parameters. In particular, one difficult problem is to know how short they can be and to construct short minimal codes. In \cite{alfarano2021three,scotti,bishnoi2023blocking} lower bounds on the length of minimal codes are proved, whereas in \cite{alfarano2023outer,bishnoi2023blocking,heger2021short} the best known upper bounds on the length of the shortest minimal codes are presented. These last are implicit existence results. Some short constructions are illustrated in \cite{alfarano2021three,bartoli2023small} and in the upcoming \cite{Expander}. More recently, minimal codes in the rank metric have been introduced \cite{alfarano2021linearcutting} together with some bounds and construction. In particular, their geometry is studied, in connection with linear sets. Such codes reveal to be useful in the construction of MRD codes \cite{bartoli2022new} or for the covering problem in the rank metric
\cite{bonini2022saturating}. Still a direct application to cryptography of these objects is missing from the party, even if rank-metric minimal codes may be used to construct minimal codes in the Hamming metric. We will show that the same holds for minimal sum-rank metric codes.\\

After recalling some main definitions and results in the preliminary Section \ref{sec:prel}, we introduce the main object of the paper in Section \ref{sec:min}: exploiting the geometry of sum-rank metric codes studied in \cite{neri2021thegeometryS}, we prove first that minimal sum-rank metric codes correspond to collections of linear sets whose union forms a strong blocking set, that is a set of points in the projective space whose intersection with every hyperplane spans the hyperplane. This allows to easily get a bound on the maximum weight of minimal sum-rank metric codes, together with a characterization of minimal MSRD codes. Standard equations allow us to prove some bounds on the parameters of minimal sum-rank metric codes (see Theorem \ref{th:bound}) presented also in their asymptotic version. We then focus on some existence results of short minimal sum-rank metric codes, obtained by implicit counting arguments. Section~\ref{sec:geodual} is devoted to a tool that we call geometric dual: we make use of the dual of $\F_q$-subspaces of $\F_{q^m}$-vector spaces studied in \cite{polverino2010linear} to build the dual of systems associated to sum-rank metric codes and we call geometric dual the code associated to these dual systems. We prove first that such object is well-defined and we show a sort of MacWilliams' relations between the generalized weight enumerators (see Theorem \ref{th:pargeomdual}). Moreover, we prove that the geometric dual is involutory. In Section~\ref{sec:minimal}, we come back to the core of the paper, which are minimal codes in the sum-rank metric. By the correspondence between sum-rank metric codes and Hamming-metric ones, we first highlight a generalization of the celebrated Ashikhmin-Barg condition, which is a sufficient condition on the weight distribution for a code to be minimal (see Theoreom \ref{th:ABcond}). All one-weight codes result to be minimal, but remarkably there are many more one-weight codes in the sum-rank metric than in the two more classical ones. After recalling three families of one-weight sum-rank metric codes introduced in \cite{neri2021thegeometryS}, we read the property of being one-weight in the geometric dual, which exchanges the role of hyperplanes and points. As a consequence, we get that partitions in scattered linear sets yield one-weight codes. Doubly extended linearized Reed-Solomon of dimension $2$, which correspond partition of the projective line in scattered linear sets, are short minimal codes (that meet the bound in Theorem \ref{th:bound}) whose geometric dual are one-weight which are short minimal codes for quadratic extensions. We use this partition of the projective line to construct a partition of higher dimension projective spaces, yielding other examples of one-weight codes. Another construction may be done with canonical subgeometries. Finally, we show that every sum-rank metric code can be extended to a one-weight code, showing that there are several examples of one-weight codes with different geometric structures. In the last part of the section, we study some two-weight codes: these can be easily obtained by considering proper subsets of mutually disjoint scattered linear sets and taking the geometric dual of the associated codes (see Theorem \ref{th:ranklists}). Thanks to the Ashikhmin-Barg condition, we have that if the number of blocks is sufficiently large, then such codes are minimal  (see Theorem \ref{th:twoweightAB}). We finally present some examples of minimal two-weight codes with two blocks and, quite remarkably, in dimension $3$ we are able to prove the minimality  by direct geometric arguments for codes not satisfying the Ashikhmin-Barg condition (see Theorem \ref{th:twoweight2blocks}). Let us point out that their associated Hamming metric codes have few weights and in some cases they are two-weights: this make them particularly interesting for several reasons including quantum codes and strongly regular graphs; see \cite{calderbank1986geometry,hu2022divisible}.

\bigskip

\section{Preliminaries}\label{sec:prel}

In this section we will briefly recall the main results of the theory of sum-rank metric codes and linear sets, which we will use in the rest of the paper.

\medskip

\subsection{Basic notions on sum-rank metric codes}

Throughout the paper,
$\mathbf{n}=(n_1,\ldots,n_t) \in \mathbb{N}^t$ denotes an ordered tuples with $n_1 \geq n_2 \geq \ldots \geq n_t$ and $N = n_1+\ldots+n_t$. We use the following compact notations for the direct sum of vector spaces 
$$\F_{q^m}^\bfn=\bigoplus_{i=1}^t\F_{q^m}^{n_i}.$$

Let start by recalling that the \textbf{rank} of a vector $v=(v_1,\ldots,v_n) \in \F_{q^m}^n$ is defined as $\rk(v)=\dim_{\fq} (\langle v_1,\ldots, v_n\rangle_{\fq})$ and the \textbf{sum-rank weight} of an element $x=(x_1 ,\ldots, x_t) \in \F_{q^m}^\bfn$ is 
$$ w(x)=\sum_{i=1}^t \rk(x_i).$$

\begin{remark}\label{rem:bridge}
    If $t=1$, then the sum-rank weight is simply the rank, whereas if $n_1=\ldots=n_t=1$, the sum-rank weight coincides with the Hamming weight. The sum-rank metric constitues then a bridge between the rank and the Hamming metrics. 
\end{remark}

We also call \textbf{rank-list} of $x=(x_1,\ldots,x_t) \in \F_{q^m}^\bfn$ the vector
\[ (\rk(x_1),\ldots,\rk(x_t)). \]
Hence, the sum-rank weight of a vector corresponds to the sum of the entries of its rank-list.

\begin{definition}
A \textbf{(linear) sum-rank metric code} $\C $ is an $\F_{q^m}$-subspace of $\F_{q^m}^{\bfn}$ endowed with the sum-rank distance defined as
\[
d(x,y)=w(x-y)=\sum_{i=1}^t \rk(x_i-y_i),
\]
where $x=(x_1 , \ldots , x_t), y=(y_1 , \ldots, y_t) \in \F_{q^m}^\bfn$. 
Let $\C \subseteq \F_{q^m}^\bfn$ be a sum-rank metric code. We will write that $\C$ is an $[\bfn,k,d]_{q^m/q}$ code (or $[\bfn,k]_{q^m/q}$ code) if $k$ is the $\F_{q^m}$-dimension of $\C$ and $d$ is its minimum distance, that is 
\[
d=d(\C)=\min\{d(x,y) \colon x, y \in \C, x \neq y  \}.
\]    
\end{definition}

Let $\C \subseteq \F_{q^m}^{\mathbf{n}}$ be a linear sum-rank metric code. Let $G=(G_1\lvert \ldots \lvert G_t) \in \F_{q^m}^{k \times N}$ be a \textbf{generator matrix} of $\C$, that is a matrix whose lines generate $\C$, with $G_1,\ldots,G_t \in \F_{q^m}^{k \times {n_i}}$. We define $\C$ to be \textbf{nondegenerate} if the columns of $G_i$ are $\fq$-linearly independent for $i\in \{1,\ldots,t\}$ (this is independent of the choice of $G$; see \cite[Definition 2.11, Proposition 2.13]{neri2021thegeometryS}).

We will only consider nondegenerate codes in this paper and this is not a restriction since we can always embed a sum-rank metric code in a smaller space in which it results to be nondegenerate, preserving its metric properties.
So, \textbf{throughout this paper we will omit the term nondegenerate and all codes considered will be nondegenerate}.

\medskip

For sum-rank metric codes the following Singleton-like bound holds; see also \cite{byrne2021fundamental}.

\begin{theorem}[{\cite[Proposition 16]{martinezpenas2018skew}}]  \label{th:Singletonbound}
     Let $\mathcal{C}$ be an $\Fmnkd$ code. Then 
     \[
     d \leq N-k+1.
     \]
\end{theorem}

\begin{definition}
An $\Fmnkd$ code is called a \textbf{Maximum Sum-Rank Distance code} (or shortly \textbf{MSRD code}) if $d=N-k+1$.
\end{definition}

The next result classifies the $\Fm$-linear isometries of $\F_{q^m}^\bfn$ equipped with the sum-rank distance, cfr. \cite[Theorem 3.7]{alfarano2021sum} and \cite[Theorem 2]{martinezpenas2021hamming}.
Before stating it, we need the following notation. 
Let $\ell:=\lvert \{n_1,\ldots,n_t\} \rvert $ and let $n_{i_1},\ldots,n_{i_{\ell}}$ be the distinct elements of $\{n_1,\ldots,n_t\}$.
By $\lambda(\mathbf{n})\in \N^{\ell}$ we will denote the vector whose entries are
\[
\lambda_j:=\lvert \{k \colon n_k=n_{i_j} \} \rvert, \ \ \ \ \mbox{for each }j\in\{1,\ldots,\ell\}.
\]
For a vector $\textbf{v}=(v_1,\ldots, v_{\ell})\in \mathbb{N}^{\ell}$, we define $$S_{\textbf{v}}=S_{v_{1}} \times \cdots \times S_{v_\ell},$$
where $S_i$ is the symmetric group of order $i$ and naturally acts on the blocks of length $i$. 
Similarly, we denote by $\GL(\textbf{v}, \F_q)$ the direct product of the general linear groups of degree $v_i$ over $\F_q$, i.e.
$$ \GL(\mathbf{v}, \F_q) = \GL(v_1, \F_q)\times \ldots\times \GL(v_t, \F_q).$$

\begin{theorem}
The group of $\F_{q^m}$-linear isometries of the space $(\F_{q^m}^\bfn,d)$  is
$$((\F_{q^m}^\ast)^{t} \times \GL(\bfn, \F_q)) \rtimes \mathcal{S}_{\lambda(\bfn)},$$
which (right)-acts as 
  \begin{equation*} (x_1 , \ldots, x_t)\cdot (\mathbf{a},A_1,\ldots, A_t,\pi) \longmapsto (a_1x_{\pi(1)}A_1 \mid \ldots \mid a_tx_{\pi(t)} A_{t}).\end{equation*}
\end{theorem}

We use the $\Fm$-linear isometries of the whole ambient space to define the equivalence of sum-rank metric codes.

\begin{definition}\label{def:equiv_codes} 
Two $\Fmnk$ sum-rank metric codes $\cC_1, \cC_2$ are \textbf{equivalent} if there is an $\Fm$-linear isometry $\phi$, such that $\phi(\cC_1)=\cC_2$. The set of equivalence classes of $\Fmnkd$ sum-rank metric codes is denoted by $\mathfrak{C}\Fmnkd$.
\end{definition}

\medskip

\subsection{The geometry of sum-rank metric codes}

We will recall now some results from \cite{neri2021thegeometryS}, on the connections between sum-rank metric codes and some sets of subspaces.

The following definition extends the notion of projective systems and $q$-systems; cfr. \cite{Randrianarisoa2020ageometric,vladut2007algebraic}.

\begin{definition} 
An $[\mathbf{n},k,d]_{q^m/q}$-\textbf{system} (or simply an $[\mathbf{n},k]_{q^m/q}$-\textbf{system}) $U$ is an ordered set $(U_1,\cdots,U_t)$, where, for any $i\in \{1,\ldots, t\}$, $U_i$ is an $\F_q$-subspace of $\F_{q^m}^k$ of dimension $n_i$, such that
$ \langle U_1, \ldots, U_t \rangle_{\F_{q^m}}=\F_{q^m}^k$ and 
$$ d=N-\max\left\{\sum_{i=1}^t\dim_{\F_q}(U_i\cap H) \mid H \textnormal{ is an $\F_{q^m}$-hyperplane of }\F_{q^m}^k\right\}.$$
Moreover, two $[\mathbf{n},k,d]_{q^m/q}$-systems $(U_1,\ldots,U_t)$ and $(V_1,\ldots, V_t)$ are \textbf{equivalent} if there exists  $\varphi\in\GL(k,\F_{q^m})$ and  $\sigma\in\mathcal{S}_t$, such that
$$ \varphi(U_i) = a_iV_{\sigma(i)},$$
for every $i\in\{1,\ldots,t\}$.
\end{definition}

We denote the set of equivalence classes of $[\mathbf{n},k,d]_{q^m/q}$-systems by $\mathfrak{U}[\mathbf{n},k,d]_{q^m/q}$.

The following result allows us to establish a connection between systems and codes.

\begin{theorem}[{\cite[Theorem 3.1]{neri2021thegeometryS}}] \label{th:connection}
Let $\C$ be an $[\bfn,k,d]_{q^m/q}$. Let $G=(G_1\lvert \ldots \lvert G_t)$ be a generator matrix of $\C$.
Let $U_i \subseteq \F_{q^m}^k$ be the $\F_q$-span of the columns of $G_i$, for $i\in \{1,\ldots,t\}$.
The sum-rank weight of an element $x G \in \C$, with $x=(x_1,\ldots,x_k) \in \F_{q^m}^k$ is
\begin{equation}\label{eq:weight}
w(x G) = N - \sum_{i=1}^t \dim_{\fq}(U_i \cap x^{\perp}),
\end{equation}
where $x^{\perp}=\{y=(y_1,\ldots,y_k) \in \F_{q^m}^k \colon \sum_{i=1}^k x_iy_i=0\}$. In particular, the minimum distance of $\C$ reads as follows
\begin{equation} \label{eq:distancedesign}
d=N- \max\left\{ \sum_{i=1}^t \dim_{\fq}(U_i \cap H)  \colon H\mbox{ is an } \F_{q^m}\mbox{-hyperplane of }\F_{q^m}^k  \right\}.
\end{equation}
So $(U_1,\ldots,U_t)$ in an $[\bfn,k,d]_{q^m/q}$-system. 
\end{theorem}

\begin{remark}\label{rk:ranklist}
Note that, as a consequence of \eqref{eq:weight}, the rank-list of a codeword $x G$ of  $\C$ is given by
\[ (n_1-\dim_{\fq}(U_1 \cap x^{\perp}),\ldots,n_t-\dim_{\fq}(U_t \cap x^{\perp})). \]
\end{remark}

As in \cite{neri2021thegeometryS}, we can then observe that there is a one-to-one  correspondence  between  equivalence  classes  of  sum-rank nondegenerate $[\mathbf{n},k,d]_{q^m/q}$ code and equivalence classes of $[\mathbf{n},k,d]_{q^m/q}$-systems via two maps
\begin{align*}
    \Psi :  \mathfrak{C}[\mathbf{n},k,d]_{q^m/q} &\to\mathfrak{U}[\mathbf{n},k,d]_{q^m/q} \\
    \Phi : \mathfrak{U}[\mathbf{n},k,d]_{q^m/q} &\to \mathfrak{C}[\mathbf{n},k,d]_{q^m/q},
\end{align*}
that act as follows.
For any $[\C]\in\mathfrak{C}[\mathbf{n},k,d]_{q^m/q}$, let $G=(G_1\lvert \ldots \lvert G_t)$ be a generator matrix of $\overline{\C}$. Then $\Psi([\C])$ is defined as the equivalence class of $[\mathbf{n},k,d]_{q^m/q}$-systems $[U]$, where $U=(U_1,\ldots,U_t)$ is defined as in Theorem \ref{th:connection}. In this case $U$ is also called a \textbf{system associated with} $\overline{\C}$. Viceversa, given $[(U_1,\ldots,U_t)]\in\mathfrak{U}[\mathbf{n},k,d]_{q^m/q}$, define $G_i$ as the matrix whose columns are an $\F_q$-basis of $U_i$ for any $i$. Then $\Phi([(U_1,\ldots,U_t)])$ is the equivalence class of the sum-rank metric code $\C$ generated by $G=(G_1\lvert \ldots \lvert G_t)$. In this case $\C$ is also called a \textbf{code associated with} $U$. See \cite{neri2021thegeometryS} for the proof that these maps are well-defined.

\medskip

\subsection{Supports}

We start by observing that a sum-rank metric code in $\F_{q^m}^{\mathbf{n}}$ can also be seen as an $\fq$-subspace in $\bigoplus_{i=1}^t \F_q^{m_i \times n_i}$.

For every $r \in \{1,\ldots,t\}$, let $\Gamma_r=(\gamma_1^{(r)},\ldots,\gamma_m^{(r)})$ be an ordered $\fq$-basis of $\F_{q^m}$, and let $\Gamma=(\Gamma_1,\ldots,\Gamma_t)$. Given   $x=(x_1, \ldots ,x_t) \in \F_{q^m}^\bfn$, with $x_i \in \F_{q^m}^{n_i}$, define the element \[\Gamma(x)=(\Gamma_1(x_1), \ldots, \Gamma_t(x_t)) \in \Pi,\] where
$$x_{r,i} = \sum_{j=1}^m \Gamma_r (x_r)_{ij}\gamma_j^{(r)}, \qquad \mbox{ for all } i \in \{1,\ldots,n_r\}.$$
In other words, the $r$-th block of $\Gamma(x)$ is the matrix expansion of the vector $x_r$ with respect to the $\fq$-basis $\Gamma_r$ of $\F_{q^m}$ and this also preserves its rank; cfr. \cite[Theorem 2.7]{neri2021thegeometryS}.

\begin{definition}
Let $x=(x_1,\dots, x_t)\in\F_{q^m}^{\mathbf{n}}$ and  $\Gamma=(\Gamma_1,\ldots,\Gamma_t)$ as above.
The \textbf{sum-rank support} of $x$ is defined as the space
$$\supp_{\mathbf{n}}(x)=(\mathrm{colsp}(\Gamma_1(x_1)), \ldots, \mathrm{colsp}(\Gamma_t(x_t))) \subseteq \fq^\bfn,$$
where $\mathrm{colsp}(A)$ is the $\fq$-span of the columns of a matrix $A$.
\end{definition}

As proved in \cite[Proposition 2.1]{alfarano2021linearcutting} for the rank-metric case, the support does not depend on the choice of $\Gamma$ and we can talk about the support of a vector without mentioning $\Gamma$. For more details see \cite{martinez2019theory}.

\medskip

\subsection{Generalized weights}

Generalized rank weights have been introduced several times with different definitions, see e.g. \cite{jurrius2017defining}, whereas the theory of sum-rank generalized weights is more recent and first introduced in \cite{camps2022optimal}.

In this paper we will deal with the definition given in \cite[Section VI]{camps2022optimal} and more precisely to the geometric equivalent, which can be derived as for the rank metric in \cite[Theorem 3.14]{alfarano2021linearcutting}. For more details we refer to \cite{JohnPaolo}.

\begin{definition}
Let $\C$ be an $[\mathbf{n},k,d]_{q^m/q}$ sum-rank metric code and let $U=(U_1,\ldots,U_t)$ be an  associated system.
For any $r \in \{1,\ldots,k\}$, the \textbf{$r$-th generalized sum-rank weight} is
\begin{equation} \label{eq:defgensumrankweight}
d_r(\C)=N- \max\left\{ \sum_{i=1}^t \dim_{\fq}(U_i \cap H)  \colon H\mbox{ is an } \F_{q^m}\mbox{-subspace of }\F_{q^m}^k\mbox{ of codimension }r  \right\}.
\end{equation}
\end{definition}

In order to keep track of the metric properties of the code, as done in \cite[Definition 4]{jurrius2017defining}, we can define the generalized sum-rank weight enumerator of a code, which extend the classical weight enumerator of a code (up to the addition of $X^N$).

\begin{definition}
Let $\C$ be an $[\mathbf{n},k,d]_{q^m/q}$ sum-rank metric code and let $U=(U_1,\ldots,U_t)$ be an associated system.
For any $r \in \{1,\ldots,k\}$, the \textbf{$r$-th generalized sum-rank weight enumerator} is
\[ W_{\C}^r(X,Y)=\sum_{w=0}^N A_w^r X^{N-w}Y^{w}, \]
where $A_w^r$ is the number of $\F_{q^m}$-subspace of $\F_{q^m}^k$ of codimension $r$ such that
\[
w=N- \sum_{i=1}^t \dim_{\fq}(U_i \cap H).
\]
\end{definition}

Clearly, the first generalized sum-rank weight enumerator corresponds with the classical weight enumerator.

\medskip

\subsection{The associated Hamming-metric codes}\label{sec:sumrankHamming}

Every sum-rank metric code can also be regarded as an Hamming-metric code as shown in \cite[Section 5.1]{neri2021thegeometryS} (see also \cite[Section 4]{alfarano2021linearcutting} for the rank-metric codes).

For a collection of multisets $(\cM_1,\mathrm{m}_{1}), \ldots, (\cM_t,\mathrm{m}_{t})$ of $\PG(k-1,q^m)$. We can define their \textbf{disjoint union} as
$$\biguplus_{i=1}^t (\cM_i,\mathrm{m}_{i})=(\cM,\mathrm{m}),$$
where $\cM=\cM_1\cup\ldots\cup \cM_t$, and $\mathrm{m}(P)=\mathrm{m}_1(P)+\ldots+\mathrm{m}_t(P)$ for every $P\in \PG(k-1,q^m)$.
For every $n$-dimensional $\F_q$-subspace $U$ of $\F_{q^m}^k$, it is possible to associate the multiset $(L_U,\mathrm{m}_U)$, where  $L_U$ is the \textbf{$\F_q$-linear set} defined by $U$ (see next subsection), that is
\[L_U=\{\langle u\rangle_{\F_{q^m}}\mid u\in U\setminus \{{ 0} \}\}\subseteq \PG(k-1,q^m),\]
and $$\mathrm{m}_U(\langle v\rangle_{\F_{q^m}})=\frac{q^{w_{L_U}(\langle v\rangle_{\F_{q^m}})}-1}{q-1}.$$
This means that the multiset $(L_U,\mathrm{m}_U)$ of $\PG(k-1,q^m)$ has size (counted with multiplicity) $\frac{q^n-1}{q-1}$. 
We can now apply this procedure to the elements of an $[\mathbf{n},k]_{q^m/q}$ system $(U_1,\ldots,U_t)$. In this way we can define the multiset 
\[
\Ext(U_1,\ldots,U_t)= \biguplus\limits_{i=1}^t (L_{U_i},\mathrm{m}_{U_i}).
\]
Then $\Ext(U_1,\ldots,U_t)$ is a multiset of points of size $\frac{q^{n_1}+\ldots +q^{n_t}-t}{q-1}$ in $\PG(k-1,q^m)$.

Hence, we can give the following definition.

\begin{definition}\label{def:hammsumrank}
Let $\C$ be a linear sum-rank $[\mathbf{n},k]_{q^m/q}$ code. Let $(U_1,\ldots,U_t)$ be a system associated with $\C$. Any code $\C^H \in \Psi(\Ext(U_1,\ldots,U_t))$ is called an \textbf{associated Hamming-metric code} to $\C$.
\end{definition}

The weight distribution of the Hamming-metric code associated with a sum-rank-metric codes can be determined as follows.
For $x \in \mathbb{F}_{q^m}^n$ denote by $w_H(x)$ the \textbf{Hamming weight} of $x$, that is the number of its non-zero components. 

\begin{proposition}[{\label{prop:weight_G_Ext}\cite[Proposition 5.6]{neri2021thegeometryS}}]
 Let $G=(G_1 \,|\, \ldots \,|\, G_t)\in \Fm^{k\times N}$ be a generator matrix of an $\Fmnk$ code, and let $v \in \Fm^k \setminus \{0\}$. 
 Denote by $G_{\Ext}\in \Fm^{k\times M}$ to be any generator matrix of a Hamming-metric code $\C^H$ in $\Psi(\Ext(U_1,\ldots,U_t))$, where $M= \frac{q^{n_1}+\ldots+q^{n_t}-t}{q-1}$.
 Then
 \[w_{\HH}(vG_{\Ext})= \sum_{i=1}^t\frac{q^{n_i}-q^{n_i-\rk(vG_i)}}{q-1}.\]
In particular, the minimum distance of $\C^{\HH}$ is given by
 $$ d(\C^{\HH})= \min_{\mathbf{r}\in \mathrm{S}(\C)}\left\{\sum_{i=1}^t\frac{q^{n_i}-q^{n_i-r_i}}{q-1}\right\},$$
 where $\mathrm{S}(\C)$ is the set of rank-lists of $\C$. 
\end{proposition}

\medskip

\subsection{Linear sets}

Let $V$ be a $k$-dimensional vector space over $\F_{q^m}$ and consider $\Lambda=\PG(V,\F_{q^m})=\PG(k-1,q^m)$.
Let $U$ be an $\fq$-subspace of $V$ of dimension $n$. Then the point-set
\[ L_U=\{\la { u} \ra_{\mathbb{F}_{q^m}} : { u}\in U\setminus \{{ 0} \}\}\subseteq \Lambda \]
is called an $\fq$-\textbf{linear set of rank $n$}.
Another important notion is the weight of a point.
Let $P=\langle v\rangle_{\F_{q^m}}$ be a point in $\Lambda$. The \textbf{weight of $P$ in $L_U$} is defined as 
\[ w_{L_U}(P)=\dim_{\fq}(U\cap \langle v\rangle_{\F_{q^m}}). \] 
An upper bound on the number of points that a linear set contains is
\begin{equation}\label{eq:card}
    |L_U| \leq \frac{q^n-1}{q-1}.
\end{equation}
Furthermore, $L_U$ is called \textbf{scattered} (and $U$ as well) if it has the maximum number $\frac{q^n-1}{q-1}$ of points, or equivalently, if all points of $L_U$ have weight one. Blokhuis and Lavrauw provided the following bound on the rank of a scattered liner set.

\begin{theorem}[{\cite{blokhuis2000scattered}\label{th:boundscattrank}}]
The rank of a scattered $\fq$-linear set in $\PG(k-1,q^m)$ is at most $\displaystyle\frac{mk}2$.
\end{theorem}

A scattered $\fq$-linear set of rank $\frac{km}2$ in $\PG(k-1,q^m)$ is said to be \textbf{maximum scattered} and $U$ is said to be a maximum scattered $\fq$-subspace as well. 

In the next result we summarize what is known on the existence of maximum scattered linear sets/subspaces.

\begin{theorem}[see \cite{blokhuis2000scattered,ball2000linear,bartoli2018maximum,csajbok2017maximum}]\label{th:existencemaxscatt}
If $mk$ is even, then there exist maximum scattered subspaces in $\F_{q^m}^k$.
\end{theorem} 

We refer to \cite{lavrauw2015field,lavrauw2016scattered,polverino2010linear,polverino2020connections,zini2021scattered} for further details on linear sets and their connections.

\bigskip

\section{Minimal sum-rank metric codes and cutting systems}\label{sec:min}

In this section we introduce the notion of sum-rank metric minimal codes and we investigate their parameters. The geometry of minimal codes have been important in order to construct and give bounds in both Hamming and rank metric (see \cite{tang2021full,alfarano2021three,alfarano2021linearcutting,bishnoi2023blocking,alfarano2023outer,heger2021short}), via the so called \textbf{strong blocking sets}. These, introduced first in \cite{davydov2011linear} in relation to saturating sets, are sets of points in the projective space such that the intersection with every hyperplane spans the hyperplane. In \cite{fancsali2014lines} strong blocking sets are referred to
as generator sets and they are constructed as union of disjoint lines. They have gained very recently a renovated interest in coding theory, since \cite{bonini2021minimal}, in which they are named \textbf{cutting blocking sets} and they are used to construct minimal codes. Quite surprisingly, they have been shown to be the geometric counterparts of minimal codes in \cite{Alfarano2020ageometric,tang2021full}.

\medskip

\subsection{Definition and first properties}

In this subsection we introduce minimal codes in the sum-rank metric and their geometry, together with some structure results.

\begin{definition}
Let $\C$ be an $[\mathbf{n},k]_{q^m/q}$ sum-rank metric code. A codeword $c \in \C$ is said \textbf{minimal} if for every $c'\in \C$ such that $\supp_{\mathbf{n}}(c')\subseteq \supp_{\mathbf{n}}(c)$ then $c'=\lambda c$ for some $\lambda \in \F_{q^m}$. We say that $\C$ is \textbf{minimal} if all of its codewords are minimal. 
\end{definition}

\begin{definition}
An $[\mathbf{n},k]_{q^m/q}$ system $(U_1,\ldots,U_t)$ is called \textbf{cutting}  if for any hyperplane $H$ of $\mathbb{F}_{q^m}^k$  
\[ \langle U_1\cap H,\ldots, U_t\cap H\rangle_{\mathbb{F}_{q^m}}=H, \]
that is, if $L_{U_1}\cup \ldots\cup L_{U_t}$ is a strong blocking set in $\PG(k-1,q^m)$.
\end{definition}

The following is a generalization of the geometric characterization of minimal codes in the Hamming and in the rank metric \cite{Alfarano2020ageometric,alfarano2021linearcutting}.

\begin{theorem}[{\cite[Corollary 10.25]{santonastaso2022subspace}}]\label{th:correspondence} 
A sum-rank metric code is minimal if and only if an associated system is cutting.
\end{theorem}

Thanks to this correspondence, we can easily prove, as in \cite[Corollary 5.9]{alfarano2021linearcutting}, a bound on the maximum weight of a minimal sum-rank metric code.

\begin{theorem}\label{th:boundmaxweight}
Let $\C$ be an $[\mathbf{n},k]_{q^m/q}$ minimal sum-rank metric code and denote by $w(\C)$ the maximum weight of the codewords in  $\C$. Then
\[w(\C)\leq N-k+1.\]
\end{theorem}
\begin{proof}
Let $(U_1,\ldots,U_t)$ be a system associated with $\C$. Since $\C$ is minimal, by Theorem \ref{th:correspondence}, for any hyperplane $H$ of $\mathbb{F}_{q^m}^k$ we have 
\[ \langle U_1\cap H,\ldots, U_t\cap H\rangle_{\mathbb{F}_{q^m}}=H, \]
which implies 
\begin{equation}\label{eq:maxweight} 
\sum_{i=1}^t \dim_{\mathbb{F}_q}(U_i\cap H)\geq k-1. 
\end{equation}
By Theorem \ref{th:connection} the maximum weight of $\C$ is
\[ w(\C)=N-\min \left\{ \sum_{i=1}^t \dim_{\mathbb{F}_q}(U_i\cap H) \colon H \text{ is an $\mathbb{F}_{q^m}$-hyperplane of $\mathbb{F}_{q^m}^k$} \right\}, \]
and by \eqref{eq:maxweight} the assertion follows.
\end{proof}

We can provide a characterization of MSRD codes which are minimal.

\begin{corollary}
An MSRD code with parameters $[\mathbf{n},k]_{q^m/q}$ is minimal if and only if it is a one-weight code with minimum distance $N-k+1$.
\end{corollary}
\begin{proof}
Let $\C$ be an MSRD, that is its minimum distance is $d=N-k+1$. 
By Theorem \ref{th:boundmaxweight}, we also know that $w(\C)\leq N-k+1$, therefore $d=w(\C)= N-k+1$. 
The converse trivially holds. 
\end{proof}

We consider now the \textbf{Standard Equations}, extending \cite[Lemma 3.6]{alfarano2021linearcutting}. Let us recall here that
$${N \choose K}_{q^m}=\prod_{i=0}^{K-1}\frac{q^N-q^i}{q^K-q^i}$$
denotes the number of $K$-dimensional subspaces of $\F_{q^m}^N$ and it is called the \textbf{Gauss binomial coefficient}.

\begin{lemma}[Standard Equations]\label{th:standeqs}
Let $U=(U_1,\ldots,U_t)$ an $[\mathbf{n},k]_{q^m/q}$-system and let
\[ \Lambda_r=\left\{ W \colon W \text{ is an $r$-dimensional $\mathbb{F}_{q^m}$-subspace of } \mathbb{F}_{q^m}^k \right\}. \]
Then 
\[ \sum_{W \in \Lambda_r, i \in \{1,\ldots,t\}} |W \cap U_i\setminus\{0\}| =(q^{n_1}+\ldots+q^{n_t}-t){k-1 \choose r-1}_{q^m}. \]
\end{lemma}
\begin{proof}
The assertion follows from the fact that for any $i \in \{1,\ldots,t\}$, \cite[Lemma 3.6]{alfarano2021linearcutting} implies
\[ \sum_{W \in \Lambda_r} |W \cap U_i\setminus\{0\}| =(q^{n_i}-1){k-1 \choose r-1}_{q^m}, \]
and
\[ \sum_{W \in \Lambda_r, i \in \{1,\ldots,t\}} |W \cap U_i\setminus\{0\}| = \sum_{i \in \{1,\ldots,t\}} \left(\sum_{W \in \Lambda_r} |W \cap U_i\setminus\{0\}| \right).\]
\end{proof}

\medskip

\subsection{Bounds on the parameters of minimal sum-rank metric codes}

By extending a minimal sum-rank metric code by adding new columns and/or blocks we get a minimal code as well.

\begin{proposition}\label{prop:extendminimal}
Let $\C$ be a minimal sum-rank metric code with parameters $[\mathbf{n},k]_{q^m/q}$. Let $\C'$ be the code generated by $G'$, where $G'$ is obtained by adding any columns or blocks to any generator matrix $G$ of $\C$. Then $\C'$ is minimal.
\end{proposition}
\begin{proof}
This is an immediate consequence of Theorem \ref{th:correspondence}.
\end{proof}

In view of Proposition \ref{prop:extendminimal}, it is natural to look for \emph{short} minimal sum-rank metric codes, when the number of blocks is given. 

The Standard Equations allow to prove the following bound on the parameters of minimal sum-rank metric codes.

\begin{theorem}\label{th:bound}
Let $\C$ be an $[\mathbf{n},k]_{q^m/q}$ minimal sum-rank metric code.\\
If $t\geq k$, then
\begin{equation}\label{eq:bound1}
(q^{n_1}+\ldots+q^{n_t}-t)(q^{m(k-1)}-1)\geq (q-1)(k-1)(q^{km}-1)
\end{equation}
If $t\leq k-1$ then
\begin{equation}\label{eq:bound2}
(q^{n_1}+\ldots+q^{n_t}-t)(q^{m(k-1)}-1)\geq t(q^{\left\lfloor \frac{k-1}{t} \right\rfloor}-1)(q^{km}-1).
\end{equation}
\end{theorem}
\begin{proof}
By Theorem \ref{th:connection}, for any hyperplane
\begin{equation}\label{eq:boundsumdims}
\sum_{i \in \{1,\ldots,t\}} \dim_{\mathbb{F}_q} (H\cap U_i)\geq N-w(\C)\geq k-1,\end{equation}
therefore
\[ \sum_{i \in \{1,\ldots,t\}} |H \cap U_i\setminus\{0\}|\geq  (q-1)(k-1). \]
It follows that 
\[ \sum_{H \in \Lambda_{k-1},  i \in\{1,\ldots,t\}} |H \cap U_i\setminus\{0\}| \geq (q-1)(k-1){k \choose 1}_{q^m}, \]
so that, by Lemma \ref{th:standeqs},
\[ (q^{n_1}+\ldots+q^{n_t}-t){k-1 \choose 1}_{q^m}\geq (q-1)(k-1){k \choose 1}_{q^m}, \]
which is \eqref{eq:bound1}.

If $t\leq k-1$, then \eqref{eq:boundsumdims} implies that 
\[ \dim_{\fq}(U_i \cap H)\geq \left\lfloor \frac{k-1}{t} \right\rfloor, \]
for any $i \in \{1,\ldots,t\}$ and for any hyperplane $H$. Arguing as before we obtain \eqref{eq:bound2}.
\end{proof}

\begin{corollary}\label{cor:boundminlegcodes}
Let $\C$ be an $[\mathbf{n},k]_{q^m/q}$ minimal sum-rank metric code.\\
If $t\geq k$, then, for large $q$,
\begin{equation}\label{eq:asbound1}
N \geq t+m+\lceil \log_q(k)\rceil,
\end{equation}
If $t\leq k-1$ then, for large $q$,
\begin{equation}\label{eq:asbound}
N\geq \left\lfloor \frac{k-1}{t} \right\rfloor+m+\lceil \log_q(t)\rceil+t-1.
\end{equation}
\end{corollary}

\begin{proof}
Noting that $q^{N-t+1}+(t-1)q\geq q^{n_1}+\ldots+q^{n_t}$ the asymptotic bounds follow.
\end{proof}

\begin{remark}
For $t=N$ (Hamming-metric case), the bound \eqref{eq:bound1} becomes 
\[N\geq \left\lceil\frac{(q^{m})^k-1}{(q^m)^{k-1}-1}\cdot (k-1)\right\rceil.\]
This last is in general slightly weaker than the known lower bound on the length of minimal codes (see \cite[Theorem 2.14]{alfarano2021three}), recently improved in \cite[Theorem 1.4.]{bishnoi2023blocking} and \cite[Theorem A]{scotti} for large $k$. Note that for $k=2$, the above bound is sharp.

For $t=1$ (rank-metric case), the bound \eqref{eq:asbound} reduces to 
\[N\geq m+k-1,\] which is exactly the bound proved for rank-metric codes \cite[Corollary 5.10]{alfarano2021linearcutting}, which is shown to be sharp for $k=2$ and for $k=3$, the last with some additional conditions on $m$ (see \cite[Theorem 6.7]{alfarano2021linearcutting} for the precise statement). 

Bound \eqref{eq:bound1} is tight for every $q$, $m$, $t=q+1$ and $k=2$ as we will show in Remark \ref{rk:partlineRScodes}.
Moreover, for $k=3$ and $q>2$, consider a code $\C$ associated with the $[\mathbf{n},3]_{q^m/q}$-system 
\[ U=(U_1,\ldots,U_t), \]
where $U_1$ is a scattered $\fq$-subspace of dimension $m+2$ (which exists under some conditions, see again \cite[Theorem 6.7]{alfarano2021linearcutting}) and $U_2,\ldots,U_t$ any $\fq$-subspaces of dimension one spanned by random nonzero vectors in $\F_{q^m}^k$.
By \cite[Theorem 6.7]{alfarano2021linearcutting} and Theorem \ref{th:correspondence}, $\C$ is a minimal sum-rank metric code with $N=m+2+t-1=m+t+1$ which gives the equality in bounds \eqref{eq:asbound1} (for $t\geq 3$) and \eqref{eq:asbound} (for $t=2$).
\end{remark}

\medskip

\subsection{Existence of minimal codes}

A first result immediately follows from \cite[Corollary 6.11]{alfarano2021linearcutting} on the existence of minimal codes in the rank metric.

\begin{proposition}\label{prop:existence}
For any $m,k\geq 2$, $t$ and $n_2,\ldots,n_t$, there exists a minimal code of parameters $[(2k+m-2,n_2,\ldots,n_t),k]_{q^m/q}$.
\end{proposition}
\begin{proof}
\cite[Corollary 6.11]{alfarano2021linearcutting} ensures the existence of a minimal $\F_{q^m}$-linear rank-metric code of length $2k+m-2$ of dimension $k$. Then we can extend such a code to a minimal sum-rank metric code via Proposition \ref{prop:extendminimal}.
\end{proof}

However, the codes described in Proposition \ref{prop:existence} are quite unbalanced, since we look only to the first block and we do not care of the rest. The following result will give a more general existence condition. We  follow the proof of \cite[Lemma 6.10]{alfarano2021linearcutting} to give a condition on the parameters which ensures the existence of a minimal sum-rank metric code. The main difference with the proof of \cite[Lemma 6.10]{alfarano2021linearcutting} consists in computing the size of the analog of the set denoted by $\mathcal{P}$ in \cite{alfarano2021linearcutting}.

\begin{theorem}
If $n_i\geq m$ for any $i \in \{1,\ldots,t\}$ and
\begin{equation}\label{eq:existence}
    \frac{(q^{mN}-1)(q^{m(N-1)}-1)}{(q^{mk}-1)(q^{m(k-1)}-1)}-\frac{1}2  \sum_{i_1,\ldots,i_t=2}^{m} \frac{1}{q^m-1}\prod_{r=1}^t{m \choose i_r}_q \prod_{j_r=0}^{i_r-1} (q^{n_r}-q^{j_r}) \left( \frac{q^{mi_r}-1}{q^m-1}-1 \right)
\end{equation} 
is positive, then there exists a linear sum-rank metric code with parameters $[\mathbf{n},k]_{q^m/q}$ which is minimal.
\end{theorem}
\begin{proof}
Denote by $\mathcal{Q}$ a set of nonzero representatives of the one-dimensional $\F_{q^m}$-subspaces of $\F_{q^m}^k$. A non-minimal linear code in $\mathbb{F}_{q^m}^{\mathbf{n}}$ is any sum-rank metric code containing an element of the following set
\[ \mathcal{P}=\{ (x,y) \in \mathcal{Q}^2 \colon x\ne y, \mathrm{supp}^{\mathrm{srk}}(x)\subseteq \mathrm{supp}^{\mathrm{srk}}(y) \text{ or } \mathrm{supp}^{\mathrm{srk}}(y)\subseteq \mathrm{supp}^{\mathrm{srk}}(x) \}. \]
Therefore, the number of minimal sum-rank metric codes in $\mathbb{F}_{q^m}^{\mathbf{n}}$ is at least
\[ {N \choose k}_{q^m}- |\mathcal{P}| {N-2 \choose k-2}_{q^m}={N-2 \choose k-2}_{q^m} \left( \frac{(q^{mN}-1)(q^{m(N-1)}-1)}{(q^{mk}-1)(q^{m(k-1)}-1)} - |\mathcal{P}| \right), \]
and hence if we prove that $\frac{(q^{mN}-1)(q^{m(N-1)}-1)}{(q^{mk}-1)(q^{m(k-1)}-1)} - |\mathcal{P}|>0$, then we ensure the existence of a minimal sum-rank metric code.
Finally we compute the size of $\mathcal{P}$ as follows:
\begin{align*}
   2 |\mathcal{P}| &=\sum_{i=1}^{tm} |\{ (x,y) \in \mathcal{Q}^2 \colon x\ne y, \mathrm{w}(y)=i,\,\, \mathrm{supp}^{\mathrm{srk}}(x)\subseteq \mathrm{supp}^{\mathrm{srk}}(y)  \}|\\
    &= \sum_{i_1,\ldots,i_t=1}^{m} \sum_{y \in \mathcal{Q},\ \ \mathrm{w}(y)=i_1+\ldots+i_t}  |\{ x \in \mathcal{Q} \colon x\ne y, \mathrm{supp}^{\mathrm{srk}}(x)\subseteq \mathrm{supp}^{\mathrm{srk}}(y)  \}|\\
    &= \sum_{i_1,\ldots,i_t=1}^{m} \sum_{y \in \mathcal{Q},\ \ \mathrm{w}(y)=i_1+\ldots+i_t}  |\{ x \in \mathcal{Q} \colon x\ne y, \mathrm{supp}^{\mathrm{rk}}(x_\ell)\subseteq \mathrm{supp}^{\mathrm{rk}}(y_\ell)\,\, \forall \ell \in \{1,\ldots,t\}  \}|\\
    &=\sum_{i_1,\ldots,i_t=2}^{m} \frac{1}{q^m-1}\prod_{r=1}^t{m \choose i_r}_q \prod_{j_r=0}^{i_r-1} (q^{n_r}-q^{j_r}) \left( \frac{q^{mi_r}-1}{q^m-1}-1 \right).
\end{align*}
\end{proof}

\begin{corollary}
If $n_i\geq m$ for any $i \in \{1,\ldots,t\}$ and
for sufficiently large $q$, if \[N\geq 2k+tm-(t+1),\]
then there exists a linear sum-rank metric code with parameters $[\mathbf{n},k]_{q^m/q}$ which is minimal.
\end{corollary}

\begin{proof}
As in the proof of \cite[Corollary 6.11]{alfarano2021linearcutting} we have the 
\eqref{eq:existence} is greater than
\[q^{2m(N-k)}-\frac{1}{2(q^m-1)^{t+1}}\sum_{i_1,\ldots,i_t=2}^{m}\prod_{r=1}^t{m \choose i_r}_q q^{i_r(m+n_r)},
\]
which is greater than (with the same arguments as in the proof of \cite[Corollary 6.11]{alfarano2021linearcutting})
\[q^{2m(N-k)}-\frac{f(q)^t}{2(q^m-1)^{t+1}}\cdot \sum_{i_1,\ldots,i_t=2}^{m}
q^{\sum_{r=1}^t 2mi_r-i_r^2+n_ri_r},
\]
where 
\[f(q):=\prod_{i=1}^{\infty}\frac{q^i}{q^i-1}.\]
Now, by simple analytic arguments, the largest exponent in the sum on the right is
$tm^2+mN$, so that, for $q$ large enough, we get
\[q^{2m(N-k)}-\frac{f(q)^t}{2}\cdot q^{tm^2+mN-mt-m}.\]
The thesis again follows by analytic arguments.
\end{proof}

\bigskip

\section{The geometric dual of sum-rank metric codes}\label{sec:geodual}

In this section we define a new operation which take an element in $\mathfrak{C}[\mathbf{n},k,d]_{q^m/q}$ and it associates to it an element in $\mathfrak{C}[\mathbf{n}',k,d]_{q^m/q}$, where $\mathbf{n}=(n_1,\ldots,n_t)$ and $\mathbf{n}'=(mk-n_1,\ldots,mk-n_t)$. It involves systems and
we call it \textbf{geometric dual}. We will show some of its properties and, in the next section, we will show how this object helps in constructing minimal sum-rank metric codes.

\medskip

\subsection{Dual of $\fq$-subspaces of $\F_{q^m}$-vector spaces}\label{sec:dual}

Let $V$ be an $\F_{q^m}$-vector space of dimension $k$ and let $\sigma \colon V \times V \rightarrow \F_{q^m}$ be any nondegenerate reflexive sesquilinear form of $V$ and consider \[
\begin{array}{cccc}
    \sigma': & V \times V & \longrightarrow & \F_q  \\
     & (x,y) & \longmapsto & \mathrm{Tr}_{q^m/q} (\sigma(x,y)),
\end{array}
\] 
where $V$ is seen as an $\fq$-vector space of dimension $mk$.
So, $\sigma'$ is a nondegenerate reflexive sesquilinear form on $V$ seen as an $\fq$-vector space of dimension $km$. Then we may consider $\perp$ and $\perp'$ as the orthogonal complement maps defined by $\sigma$ and $\sigma'$, respectively. For an $\F_q$-subspace $U$ of  $V$ of dimension $n$, the $\F_q$-subspace $U^{\perp'}$ is the \textbf{dual} (with respect to $\sigma'$) of $U$, which has dimension $km-n$; see \cite{polverino2010linear,lunardon2008translation}.

An important property that $\sigma'$ satisfies is that the dual of an $\F_{q^m}$-subspace $W$ of $V$ is an $\F_{q^m}$-subspace as well and $W^{\perp'}=W^\perp$. Moreover, the following result will be widely used in the paper.

\begin{proposition}[{\cite[Property 2.6]{polverino2010linear}}]\label{prop:weightdual}
Let $U$ be an $\fq$-subspace of $V$ and $W$ be an $\F_{q^m}$-subspace of $V$.
Then 
\[ \dim_{\fq}(U^{\perp'}\cap W^{\perp})=\dim_{\fq}(U\cap W)+\dim_{\fq}(V)-\dim_{\fq}(U)-\dim_{\fq}(W). \]
\end{proposition}

In \cite[Proposition 2.5]{polverino2010linear}, it has been proved that if we replace $\sigma$ by another form with the same properties, with this procedure we obtain an $\fq$-subspace of $V$ which turns out to be $\mathrm{\Gamma L}(k,q^m)$-equivalent to $U^{\perp'}$.
For this reason, from now on we will assume that $\sigma$ is fixed and we will just write $U^{\perp'}$, without mentioning the form $\sigma$ used.

\medskip

\subsection{Geometric dual}

We are now ready to give the definition of geometric dual of an $\F_{q^m}$-linear sum-rank metric code, taking into account the dual described in Section \ref{sec:dual}.

\begin{definition}
Let $\C$ be an $[\bfn,k,d]_{q^m/q}$ and let $U=(U_1,\ldots,U_t)$ be an associated system to $\C$ with the property that $U_1\cap \ldots \cap U_t$ does not contain a $1$-dimensional $\F_{q^m}$-subspace.
Then a \textbf{geometric dual} $\C^{\perp_{\mathcal{G}}}$ of $\C$ (with respect to $\perp'$) is defined as $\C'$, where $\C'$ is any code associated with the system $(U_1^{\perp},\ldots,U_t^{\perp})$.
\end{definition}

\begin{remark}
In the above definition we need that $U_1\cap \ldots \cap U_t$ does not contain a $1$-dimensional $\F_{q^m}$-subspace, otherwise $(U_1^{\perp},\ldots,U_t^{\perp})$ would not be a system.
Indeed, suppose that $(U_1^{\perp},\ldots,U_t^{\perp})$ is not a system, then 
\[ \langle U_1^{\perp'},\ldots, U_t^{\perp'} \rangle_{\F_{q^m}}\subseteq H, \]
where $H$ is an $\F_{q^m}$-hyperplane of $\F_{q^m}^k$. This implies that 
\[ U_1^{\perp'}+\ldots+U_t^{\perp'}\subseteq H, \]
and by duality
\[ U_1\cap\ldots\cap U_t\supseteq H^{\perp}, \]
a contradiction since $\dim_{\F_{q^m}}(H^{\perp})=1$.
\end{remark}

We will now prove that the geometric dual of a linear sum-rank metric code is well-defined and we will give a relation on the weight distributions among the two codes.

\begin{figure}[ht]
    \centerline{
\begin{xy}
     (20,20)*+{[\C]\in\mathfrak{C}[\mathbf{n},k]_{q^m/q}}="a"; 
    (80,20)*+{[(U_1,\ldots,U_t)]\in \mathfrak{U}[\mathbf{n},k]_{q^m/q}}="b";
     (80, 0)*+{[(U_1^{\perp'},\ldots,U_t^{\perp'})]\in \mathfrak{U}[\mathbf{n}',k]_{q^m/q}}="c";
     (20,0)*+{[\C^{\perp_{\mathcal{G}}}]\in\mathfrak{C}[\mathbf{n}',k]_{q^m/q}}="d"; 
     {\ar         "a";"b"}?*!/_8pt/{\Psi\hspace{1cm}};
     {\ar    "b";"c"}?*!/^8pt/{\perp'};
     {\ar    "c";"d"}?*!/^6pt/{\Phi\hspace{0.6cm}};
\end{xy}}
\caption{The geometric dual with $\mathbf{n}=(n_1,\ldots,n_t)$ and $\mathbf{n}'=(mk-n_1,\ldots,mk-n_t)$.}\label{fig:cons2}
\end{figure}

\begin{theorem}\label{th:pargeomdual}
Let $\C$ be an $[\bfn,k,d]_{q^m/q}$ sum-rank metric code, with $\bfn=(n_1,\ldots,n_t)$, and let $U=(U_1,\ldots,U_t)$ be an associated system to $\C$.
Assume that $U_1\cap \ldots\cap U_t$ does not contain any $1$-dimensional $\mathbb{F}_{q^m}$-subspace of $\mathbb{F}_{q^m}^k$. 
Then, up to equivalence, a geometric dual $\C^{\perp_{\mathcal{G}}}$ of $\C$ does not depend on the choice of the associated system and on the choice of code in $[\C]$, hence $\perp_{\mathcal{G}}$ is well-defined.
The parameters of $\C^{\perp_{\mathcal{G}}}$ are $[(km-n_1,\ldots,km-n_t),k]_{q^m/q}$.
The generalized weight enumerators of $\C$ and $\C^{\perp_{\mathcal{G}}}$ are related as follows
\[w_{\C}^{k-r}(X,Y)= X^{tmk-trm-N}Y^{trm-N}w_{\C^{\perp_{\mathcal{G}}}}^{r}(X,Y). \]
In particular,
\[ d_1(\C)=d_{k-1}(\C^{\perp_{\mathcal{G}}})+N-t(k-1)m. \]
\end{theorem}
\begin{proof}
We now prove that $\perp_{\mathcal{G}}$ is well-defined on the equivalence classes of linear sum-rank metric codes, that is, the geometric dual of equivalent codes are equivalent (note that this also implies that the geometric dual does not depend on the choice of the system).
Let $\C'$ be a code equivalent to $\C$ and let $U'=(U_1',\ldots,U_t')$ be an associated system to $\C'$. Then $U$ and $U'$ are equivalent systems and hence there exist $\rho \in \mathcal{S}_t$, $a_1,\ldots,a_t \in \F_{q^m}^*$ such that for every $i \in \{1,\ldots,t\}$
\[ U_i'=a_i U_{\rho(i)}. \]
Note that this also implies that $U_1'\cap \ldots\cap U_t'$ does not contain any $1$-dimensional $\mathbb{F}_{q^m}$-subspace of $\mathbb{F}_{q^m}^k$, since the $U_i$'s satisfy this condition.
Then for every $i \in \{1,\ldots,t\}$
\[(U_i')^{\perp'}=\{v \in \mathbb{F}_{q^m}^k \colon \sigma'(v,a_i u)=0\,\,\forall u \in U_{\rho(i)}\}=a_i^{-1} U_{\rho(i)}^{\perp'},\]
since $\sigma'(v,a_i u)=\mathrm{Tr}_{q^m/q}(a_i\sigma (v,u))=\sigma'(a_iv, u)$.\\
This implies that the systems $(U_1^{\perp'},\ldots,U_t^{\perp'})$ and $(U_1'^{\perp'},\ldots,U_t'^{\perp'})$ are equivalent and hence $\C^{\perp_{\mathcal{G}}}$ and $\C'^{\perp_{\mathcal{G}}}$ are equivalent as well.
So, we proved that $\perp_{\mathcal{G}}$ is well-defined.

Now, observe that the code $\C^{\perp_{\mathcal{G}}}$ has dimension $k$ since $\perp'$ does not change the dimension of the ambient space of the $U_i$'s. Since $\dim_{\fq}(U_i^{\perp'})=mk-\dim_{\fq}(U_i)$ for any $i$, it follows that $\C^{\perp_{\mathcal{G}}}$ has parameters $[(km-n_1,\ldots,km-n_t),k]$.
To determine the $r$-th generalized sum-rank weight enumerator of $\C$ we need to compute
\[ N-\sum_{i=1}^t\dim_{\fq}(U_i\cap W), \]
for any $\F_{q^m}$-subspace $W$ of $\F_{q^m}^k$ of dimension $r$, where $N=n_1+\ldots+n_t$.
By Proposition \ref{prop:weightdual}, we obtain that
\[  N-\sum_{i=1}^t\dim_{\fq}(U_i\cap W)=N-\sum_{i=1}^t\dim_{\fq}(U_i^{\perp'}\cap W^{\perp})+tmk-N-trm=\]
\[ tmk-N -\sum_{i=1}^t\dim_{\fq}(U_i^{\perp'}\cap W^{\perp})+N-trm,
\]
which correspond to a weight appearing in the $(k-r)$-th generalized sum-rank weight enumerator of $\C^{\perp_{\mathcal{G}}}$ plus $N-trm$.
Therefore, if 
\[ W_{\C}^{k-r}(X,Y)=\sum_{w=0}^N A_w X^{N-w}Y^{w} \,\,\,\mbox{and}\,\,\, W_{\C^{\perp_{\mathcal{G}}}}^{r}(X,Y)=\sum_{w=0}^N B_w X^{tmk-N-w}Y^{w}, \]
we have that 
\[ A_w=B_{w-N+trm}, \]
for any $w$.
Therefore,
\begin{align*}
    W_{\C^{\perp_{\mathcal{G}}}}^{r}(X,Y) & =\sum_{w'=0}^N B_{w'}X^{tmk-N-w'}Y^{w'}\\ &=\sum_{w=0}^N B_{w-N+trm} X^{tmk-w-trm}Y^{w-N+trm} \\
 &= \sum_{w=0}^N A_w X^{tmk-trm-w}Y^{w-N+trm},
\end{align*} 
that is the assertion.
\end{proof}

\begin{remark}\label{rk:ranklistdual}
With the notation of Theorem \ref{th:pargeomdual} and by denoting $G$ a generator matrix of $\C$, by Remark \ref{rk:ranklist} we have that the rank list of the codeword $xG$ is
\[ (n_1-\dim_{\fq}(U_1\cap x^\perp),\ldots,n_1-\dim_{\fq}(U_t\cap x^\perp)), \]
which, by Proposition \ref{prop:weightdual}, is equal to
\[ (m-\dim_{\fq}(U_1^{\perp'}\cap \langle x\rangle_{\F_{q^m}}),\ldots,m-\dim_{\fq}(U_t^{\perp'}\cap \langle x\rangle_{\F_{q^m}})). \]
\end{remark}

\begin{proposition}
The geometric dual is involutory.
\end{proposition}
\begin{proof}
This immediately follows by the definition of geometric dual and from the fact that $\perp'$ is involutory.
\end{proof}

\begin{remark}
The geometric dual operation can be applied to both Hamming and rank metric. In the first case, the geometric dual of a Hamming-metric code of length $n$ and dimension $k$ will give a sum-rank metric code with parameters $[(mk-1,\ldots,mk-1),k]_{q^m/q}$, which is far from being an Hamming-metric code. For rank-metric codes with parameters $[n,k]_{q^m/q}$, the geometric dual gives a rank-metric code as well, with parameters $[mk-n,k]_{q^m/q}$.
\end{remark}

\bigskip

\section{Minimal codes with few weights}\label{sec:minimal}

In this section we will mainly deal with explicit construction of minimal codes in the sum-rank metric. We will make extensive use of the geometric dual introduced above.

As in the Hamming and in the rank metrics (see e.g. \cite{alfarano2021linearcutting}), all the one-weight sum-rank metric codes are minimal.

\begin{proposition}[{\cite[Proposition 10.26]{santonastaso2022subspace}}] \label{prop:oneweightimpliesminimal}
Let $\C$ be an $[\mathbf{n},k]_{q^m/q}$ sum-rank metric. If all the codewords of $\C$ have the same sum-rank metric weight then $\C$ is a minimal sum-rank metric code.
\end{proposition}

The main difference is that in Hamming and rank metrics, simplex codes are essentially the only  one-weight codes; see \cite{bonisoli1983every,Randrianarisoa2020ageometric}. In the sum-rank metric, we have more examples as we will see later on and as it has been proved in \cite{neri2021thegeometryS}.

Another way to get examples, is to obtain information on minimal codes in sum-rank metric by looking at the associated Hamming-metric codes, cfr. Section \ref{sec:sumrankHamming}.

\begin{proposition}[{\cite[Corollary 10.28]{santonastaso2022subspace}}]
Let $\C$ be an $[\mathbf{n},k]_{q^m/q}$ sum-rank metric code. Then $\C$ is minimal if and only if any associated Hamming-metric code $\C^H$ is minimal.
\end{proposition}

Thanks to the above proposition, we can use some conditions proved in the Hamming-metric to ensure that a Hamming-metric code is minimal. More precisely, we will now describe a generalization of the celebrated Ashikhmin-Barg condition (see \cite[Lemma 2.1]{ashikhmin1998minimal}).

\begin{theorem}[Ashikhmin-Barg condition for sum-rank metric codes]\label{th:ABcond}
Let $\C$ be an $[\mathbf{n},k]_{q^m/q}$ sum-rank metric code.
Denote by
\[ \delta_{\max}=\max_{r \in S(\C)} \left\{ \frac{q^{n_i}-q^{n_i-r_i}}{q-1} \right\} \]
and 
\[ \delta_{\min}=\min_{r \in S(\C)} \left\{ \frac{q^{n_i}-q^{n_i-r_i}}{q-1} \right\}, \]
where $S(\C)$ is the set of rank-lists of $\C$.
If 
\[ \frac{\delta_{\max}}{\delta_{\min}}< \frac{q^m}{q^m-1}, \]
then the code $\C$ is minimal.
\end{theorem}
\begin{proof}
Let $\C^{\mathrm{H}}$ be an associated Hamming-metric code to $\C$.
The minimum distance of $\C^{\mathrm{H}}$ is $\delta_{\min}$ and its maximum weight is $\delta_{\max}$.
By \cite[Lemma 2.1]{ashikhmin1998minimal} and because of the assumptions on $\delta_{\min}$ and $\delta_{\max}$, the code $\C^{\mathrm{H}}$ is minimal.
Since the sum-rank metric code $\C$ is minimal if and only if $\C^{\mathrm{H}}$ is a minimal code in the Hamming-metric, the assertion is proved.
\end{proof}

If a sum-rank metric code satisfies the assumption of Theorem \ref{th:ABcond}, we say that it satisfies the \textbf{AB-condition}.

\medskip

In the following subsections, we will first see old and new constructions of one-weight sum-rank metric codes (proving that this is a very large family) and then we will show examples of minimal codes with few weights, where some of them satisfy the AB-condition and some of them do not.

\medskip

\subsection{Sum-rank one-weight codes}

Some constructions of one-weight codes have been given in \cite{neri2021thegeometryS}, which can be divided in three families:

\begin{itemize}
    \item orbital construction (extending the simplex code);
    \item doubly extended linearized Reed-Solomon;
    \item linear sets construction.
\end{itemize}

In particular, the last two constructions give $2$-dimensional one-weight codes in the sum-rank metric, whereas the first one give constructions of any dimension.
In \cite{neri2021thegeometryS}, the authors also showed that the last two families cannot be obtained from the orbital construction. It is natural to ask whether or not there are examples of one-weight codes which cannot be obtained from the orbital constructions also for larger dimensions.

We start by recalling the orbital construction.

Let $\mathcal{G}$ be a subgroup of $\GL(k,q^m)$ and consider the action $\phi_{\mathcal{G}}$  of $\mathcal{G}$ on $\F_{q^m}^k\setminus\{0\}$, that is
$$ \begin{array}{rccl} \phi_{\mathcal{G}} : & \mathcal{G}\times (\F_{q^m}^k\setminus\{0\}) & \longrightarrow & \F_{q^m}^k\setminus\{0\}\\
&(A,v)& \longmapsto & vA. 
\end{array}$$
For any $n$ and $r$ such that $r$ divides $m$, this action naturally induces an action also on the $n$-dimensional $\F_{q^r}$-subspaces of $\F_{q^m}^k$ with kernel $ \mathcal{G} \cap \mathbb{D}_{q^r}$, where $\mathbb{D}_{q^r}=\{\alpha I_k \colon \alpha \in \F_{q^r}^*\}$. In order to get a shorter code, we can consider the action of the group $\overline{\mathcal{G}}=\mathcal{G}/(\mathcal{G} \cap \mathbb{D}_{q^r})$ on the $n$-dimensional $\F_{q^r}$-subspaces of $\F_{q^m}^k$, that we denote by $\phi_{\mathcal{G}}^{r,n}$. Finally, we say that $\mathcal{G} \leq \GL(k,q^m)$ is transitive if the action $\phi_{\mathcal{G}}^{m,1}$ is transitive;
see \cite[Section 6.1]{neri2021thegeometryS} for a more detailed discussion.

\begin{construction}[Orbital construction \cite{neri2021thegeometryS}]\label{con:cuttingorbital}
Let $U$ be an $\F_q$-subspace of $\F_{q^m}^k$ of dimension $n$ over $\F_q$. Let $\mathcal{G} \leq \GL(k,q^m)$ be a transitive subgroup and let $\mathcal{O}=(\phi_{\mathcal{G}}^{1,n}(A,U))_{A \in \overline{\mathcal{G}}}$ be the orbit (counting possible repetition) of the action of $\phi_{\mathcal{G}}^{1,n}$. When $\mathcal{G}$ is the Singer subgroup of $\mathrm{GL}(k,q^m)$ we call the orbit $\mathcal{O}$ an  $n$-\textbf{simplex}.
A sum-rank metric code associated with the system $\mathcal{O}=(U_1,\ldots,U_t)$ is an $[(n,\ldots,n),k]_{q^m/q}$ one-weight sum-rank metric code. 
\end{construction}

\begin{remark}
Let $\mathcal{O}$ and $\mathcal{O}'$ two distinct orbits as in the above construction, the system obtained by plugging together $\mathcal{O}$ and $\mathcal{O}'$ gives rise to a one-weight code and the considered system is union of orbits. So, in order to provide a construction of a sum-rank metric code $\C$ which is one-weight but does not arise from the orbital construction, we need to show that a system associated with $\C$ cannot be obtained as union of orbits.
\end{remark}

We now read the property of being one-weight in the geometric dual of the code.

\begin{proposition}\label{prop:oneweightviageometricdual}
Let $\C$ be an $[\mathbf{n},k]_{q^m/q}$ sum-rank metric code and let $(U_1,\ldots,U_t)$ be an associated system. Assume that $U_1\cap \ldots\cap U_t$ does not contain any $1$-dimensional $\mathbb{F}_{q^m}$-subspace of $\mathbb{F}_{q^m}^k$. Then $\C$ is a one-weight code if and only if $\C^{\perp_{\mathcal{G}}}$ is a one-weight code with respect to the generalized weights of order $k-1$, that is 
\begin{equation}\label{eq:k-1oneweight}
d_{k-1}(\C^{\perp_{\mathcal{G}}})=N- \sum_{i=1}^t \dim_{\F_q} (U_i^{\perp'}\cap \langle w \rangle_{\mathbb{F}_{q^m}}),
\end{equation}
for any $w \in \F_{q^m}^k\setminus\{0\}$.
In particular, if $d_1(\C) \ne tm$ then
\[ \bigcup_{i=1}^t L_{U_i^{\perp'}}=\PG(k-1,q^m). \]

\end{proposition}
\begin{proof}
The code $\C$ is one-weight if and only if $w_{\C}^1(X,Y)$ only presents one monomial, that is
\[w_{\C}^1(X,Y)= (|\C|-1) X^{N-d_1(\C)}Y^{d_1(\C)}. \]
By Theorem \ref{th:pargeomdual}, this happens if and only if $w_{\C^{\perp_{\mathcal{G}}}}^{k-1}(X,Y)$ presents only one monomial, and hence the first part of the assertion. 
For the last part, observe that if $\C$ is one-weight, by the first part of the assertion we have
\begin{equation}\label{eq:weightpointentirespace} 
\sum_{i=1}^t\dim_{\fq}(U_i^{\perp'}\cap \langle w\rangle_{\F_{q^m}})=\sum_{i=1}^t w_{L_{U_i^{\perp'}}}(\langle w\rangle_{\F_{q^m}})=tmk-N-d_{k-1}(\C^{\perp_{\mathcal{G}}}), 
\end{equation}
for any $w \in \F_{q^m}^k\setminus\{0\}$.
Suppose that $\bigcup_{i=1}^t L_{U_i^{\perp'}}\ne \PG(k-1,q^m)$. Then, because of \eqref{eq:weightpointentirespace}, we can only have that  $tmk-N=d_{k-1}(\C^{\perp_{\mathcal{G}}})$.
Theorem \ref{th:pargeomdual} implies also that
\[ d_1(\C)=d_{k-1}(\C^{\perp_{\mathcal{G}}})+N-t(k-1)m, \]
and hence $d_1(\C)=tm$, a contradiction.
\end{proof}

\begin{remark}\label{rk:partlineRScodes}
By Proposition \ref{prop:oneweightviageometricdual}, and more precisely from \eqref{eq:k-1oneweight}, a partition in scattered linear set gives a one-weight code. Indeed, the construction of doubly extended linearized Reed-Solomon code with parameters $[(m,\ldots,m,1,1),2]_{q^m/q}$ can be read via an associated system $(U_1,\ldots,U_{q+1})$ with the following properties:
\begin{itemize}
    \item $\dim_{\fq}(U_i)=m$ for every $i\in \{1,\ldots,q-1\}$;
    \item $\dim_{\fq}(U_q)=\dim_{\fq}(U_{q+1})=1$;
    \item $U_i$'s are scattered $\fq$-subspaces of $\F_{q^m}^2$.
\end{itemize}
Since in $\F_{q^m}^2$ the hyperplanes coincide with the $1$-dimensional $\F_{q^m}$-subspaces of $\F_{q^m}^2$ and since the code is MSRD, we have
\[ \sum_{i=1}^{q+1} \dim_{\fq}(U_i\cap \langle w \rangle_{\F_{q^m}})=1, \]
for any $w \in \mathbb{F}_{q^m}^2\setminus\{0\}$.
Therefore, these codes are one-weight codes and by Proposition \ref{prop:oneweightimpliesminimal} they are also minimal codes.
Moreover, they meet the lower bound \eqref{eq:bound1}.
By Proposition \ref{prop:oneweightviageometricdual}, the geometric dual of a doubly extended linearized Reed-Solomon code is a one-weight code with parameters $$[(\underbrace{m,\ldots,m}_{q-1\text{ times}},2m-1,2m-1),2]_{q^m/q}.$$
Note that for $m=2$, such code meets the lower bound \eqref{eq:bound1}. We give its generator matrix in the next remark.

In particular, a partition in scattered subspaces/linear sets gives via Proposition \ref{prop:oneweightviageometricdual} a
one-weight code with nonzero weight equals to $N-t(k-1)m+1$. 
\end{remark}

\begin{remark}\label{rem:genmatdbrscodes}
We can determine a generator matrix for the geometric dual of a $2$-dimensional doubly extended linearized Reed-Solomon code.
Let $\alpha_1,\ldots,\alpha_{q-1}\in \mathbb{F}_{q^m}$ having pairwise distinct norm over $\fq$ and define
\[U_i=\{(x,a_i x^q) \colon x \in \F_{q^m}\},\]
for any $i \in \{1,\ldots,q-1\}$.
Consider the following sesquilinear form
\[\sigma' \colon ((x,y),(z,t)) \in \mathbb{F}_{q^m}^2 \mapsto \mathrm{Tr}_{q^m/q}(xt-yz) \in \fq.\]
Then
\[U_i^{\perp'}=\{ (y^q,a_i^{q^{n-1}}y) \colon y \in \F_{q^m} \}=\{ (y,a_i^{q^{n-1}}y^{q^{n-1}}) \colon y \in \F_{q^m} \},\]
for any $i$, since 
\[ \sigma'((x,a_i x^q),(y^q,a_i^{q^{n-1}}y))=\mathrm{Tr}_{q^m/q}(a_i^{q^{n-1}}xy-a_ix^q y^q)=0, \]
for every $x,y \in \mathbb{F}_{q^m}$.
By definition of $\sigma'$ we also have that
\[\langle (1,0)\rangle_{\fq}^{\perp'}=\{ (\alpha,\beta) \colon \alpha,\beta \in \F_{q^m}\,\,\text{and}\,\,\mathrm{Tr}_{q^m/q}(\beta)=0 \}\]
and 
\[\langle (0,1)\rangle_{\fq}^{\perp'}=\{ (\alpha,\beta) \colon \alpha,\beta \in \F_{q^m}\,\,\text{and}\,\,\mathrm{Tr}_{q^m/q}(\alpha)=0 \}.\]
Therefore, if $B=\{b_1,\ldots,b_m\}$ is an $\fq$-basis of $\fqm$ and $C=\{c_1,\ldots,c_{m-1}\}$ is an $\fq$-basis of $\ker(\mathrm{Tr}_{q^m/q})$, then a generator matrix for the geometric dual of a $2$-dimensional linearized Reed-Solomon code is as follows:
\begin{itemize}
    \item the $i$-th blocks has as $j$-th column $(b_j^q,a_i b_j)$, for $i \in \{1,\ldots,q-1\}$ and $j \in \{1,\ldots,m\}$;
    \item the $q$-th block has as $j$-th column $(b_j,0)$ if $j \in \{1,\ldots,m\}$ and $(0,c_{j-m})$ if $j \in \{m+1,\ldots,2m-1\}$;
    \item the last block has as $j$-th column $(0,b_j)$ if $j \in \{1,\ldots,m\}$ and $(c_{j-m},0)$ if $j \in \{m+1,\ldots,2m-1\}$.
\end{itemize}
\end{remark}

In the Remark \ref{rk:partlineRScodes}, we have seen that we can find $q-1$ maximum scattered $\fq$-linear set in $\PG(1,q^m)$ which are pairwise disjoint and cover $\PG(1,q^m)$ except for two points (which can be arbitrarily chosen).

We can use this fact to construct partition in scattered $\fq$-linear sets for higher dimensions. 

\begin{construction}[partition in scattered linear sets]
We start with the plane: consider $P=\langle v\rangle_{\F_{q^m}}$ a point in $\PG(2,q^m)$ and a line $\ell=\PG(W,\F_{q^m})$ not passing through $P$.
Denote by 
\[\ell_1=\PG(W_1,\F_{q^m}),\ldots,\ell_{q^m+1}=\PG(W_{q^m+1},\F_{q^m})\] the lines through $P$ and denote by $Q_i=\ell \cap \ell_i=\langle v_i\rangle_{\F_{q^m}}$ for any $i \in \{1,\ldots,q^m+1\}$.
For any $i \in \{1,\ldots,q^m+1\}$ consider
\[ U_{i,1},\ldots,U_{i,q-1} \]
maximum scattered $\fq$-subspaces of $W_i$ for which the associated linear sets form a partition of $\ell_i\setminus \{P,Q_i\}$, which exists because of Remark \ref{rk:partlineRScodes}.
Consider 
\[ U_{q^m+2,1},\ldots,U_{q^m+2,q-1} \]
maximum scattered $\fq$-subspaces of $W$ for which the associated linear sets form a partition of $\ell\setminus \{Q_1,Q_2\}$. 
The $\fq$-linear sets associated with $U_{i,j}$'s together with $\langle v \rangle_{\fq}$,  $\langle v_1 \rangle_{\fq}$ and $\langle v_2 \rangle_{\fq}$ give a partition in scattered $\fq$-linear sets of $\PG(2,q^m)$.
This gives a code of parameters
\[[(\underbrace{m,\ldots,m}_{(q-1)(q^m+2)\text{ times}},1,1,1),3]_{q^m/q}\]
whose geometric dual has parameters
\[[(\underbrace{2m,\ldots,2m}_{(q-1)(q^m+2)\text{ times}},3m-1,3m-1,3m-1),3]_{q^m/q}.\]

For larger dimension: suppose that we have $U_1,\ldots,U_r$ scattered $\fq$-subspaces with the property that the associated linear sets cover a projective space of dimension $k-1$. 
Consider a point $P=\langle v \rangle_{\F_{q^m}} \in \PG(k-1,q^m)$ and a hyperplane $H$ not passing through $P$, then we can proceed to cover with scattered linear sets all of the lines through $P$, except for the point $P$ and the intersection of the line with $H$. As we have done for the plane, adding the $U_i$'s and $\langle v \rangle_{\F_{q}}$ to these subspaces, we obtain a family of scattered subspace covering the entire space.
\end{construction}

The above construction contains a large number of blocks. It is possible, under certain restrictions, to consider a smaller number of subspaces with the use of canonical subgeometries. 

A \textbf{canonical subgeometry} of $\PG(k-1,q^m)$ is any $\PG(k-1,q)$ which is embedded in $\PG(k-1,q^m)$.
The following result gives condition on $k$ and $m$ which allows us to construct a partition of $\PG(k-1,q^m)$ in canonical subgeometries.

\begin{theorem}[{\cite[Theorem 4.29]{hirschfeld1998projective}}]\label{th:partionsubg}
There exists a partition of $\PG(k-1,q^m)$ into canonical subgeometries if and only if $\gcd(k,m)=1$.
\end{theorem}

\begin{remark}
Suppose that $\gcd(k,m)=1$ and let 
\[U_1,\ldots,U_t\]
be $\fq$-subspaces of $\F_{q^m}$ of dimension $k$ such that $L_{U_1},\ldots,L_{U_t}$ a partition of $\PG(k-1,q^m)$ into canonical subgeometries, with $t=\frac{(q^{mk}-1)(q-1)}{(q^m-1)(q^k-1)}$. A code associated with $(U_1,\ldots,U_t)$ has parameters $[(k,\ldots,k),k,(t-1)k]_{q^m/q}$ and it is one-weight.
Its geometric dual has parameters $[(k(m-1),\ldots,k(m-1)),k,t(k-1)]_{q^m/q}$, via Theorem \ref{th:pargeomdual}.
\end{remark}

\medskip

\subsubsection{Lift construction}

In this section we describe a procedure to construct one-weight sum-rank metric codes starting from any sum-rank metric code, extending the construction described in \cite[Section 7.2]{neri2021thegeometryS}. 

Let $U_1,\ldots,U_t$ be $\F_q$-subspaces in $\F_{q^m}^k$ and define
\[ M=\max\left\{ \sum_{i=1}^t w_{L_{U_i}}(P) \mid  P \in \PG(k-1,q^m)\right\}. \]

\noindent Define $\mathcal{M}(U_1,\ldots,U_t)$ the \textbf{lift} of $U_1,\ldots,U_t$ as a vector of $\fq$-subspaces whose entries are
\begin{itemize}
    \item $U_1,\ldots,U_t$;
    \item $c$ copies of $\langle v \rangle_{\F_{q^m}}$ and a $d$-dimensional subspace of $\langle v \rangle_{\F_{q^m}}$, for any $P=\langle v \rangle_{\F_{q^m}}\in \PG(k-1,q^m)$, where $M-\sum_{i=1}^t w_{L_{U_i}}(P)=c \cdot m+d$ with $c,d \in \mathbb{N}$ and $d<m$.
\end{itemize}

\noindent Clearly, by construction
\[ \sum_{U \in\mathcal{M}(U_1,\ldots,U_t)} w_{L_U}(P)=M, \]
for every point $P\in \PG(k-1,q^m)$.

So, consider $\C \in \Phi ([\mathcal{M}(U_1,\ldots,U_t)])$ then by  applying Proposition \ref{prop:oneweightviageometricdual} we obtain that $\C^{\perp_{\mathcal{G}}}$ is a linear one-weight sum-rank metric code of dimension $k$.
Any code in $\Phi([\mathcal{M}(U_1,\ldots,U_t)])$ will be called the \textbf{lifted code of} $U_1,\ldots,U_t$.

In the next result we show that we can construct a one-weight sum-rank metric code starting from any sum-rank metric code.

\begin{theorem}\label{th:everysumrankisoneweight}
Every sum-rank metric code can be extended to a one-weight code.
\end{theorem}
\begin{proof}
Let $(U_1,\ldots,U_t)$ an associated system to $\C$ and $G=(G_1|\ldots|G_t)$ a generator matrix of $\C$ such that the column span of $G_i$ is $U_i$ for any $i$. Consider the lift \[\mathcal{M}(U_1^{\perp'},\ldots,U_t^{\perp'})=(U_1^{\perp'},\ldots,U_t^{\perp'},W_{t+1},\ldots,W_s).  \]
Let $\mathcal{D} \in \Phi([\mathcal{M}(U_1^{\perp'},\ldots,U_t^{\perp'})])$, then by duality $U_1^{\perp'}\cap\ldots\cap U_t^{\perp'}$ does not contain any one-dimensional $\F_{q^m}$-subspace and hence by Theorem \ref{th:pargeomdual} $\mathcal{D}^{\perp_{\mathcal{G}}}$ is a one-weight code and a system associated with $\mathcal{D}^{\perp_{\mathcal{G}}}$ is 
\[ (U_1,\ldots,U_t,W_{t+1}^{\perp'},\ldots,W_s^{\perp'}), \]
and hence it is an extension of $\C$, since a generator matrix of $\mathcal{D}^{\perp_{\mathcal{G}}}$ has the following shape
\[ (G_1|\ldots|G_t|G_{t+1}|\ldots|G_s), \]
where the the column span of $G_i$ is $W_i^{\perp'}$ for any $i \in \{t+1,\ldots,s\}$.
\end{proof}

In the following result we prove that there are linear sum-rank metric codes which cannot be obtain from the orbital construction for any possible value of the dimension, already proved in \cite{neri2021thegeometryS} in the two-dimensional case.

\begin{theorem}
For every $k$, there are one-weight sum-rank metric codes which are not equivalent to a sum-rank metric code obtained from the orbital construction.
\end{theorem}
\begin{proof}
Consider $\C$ a $[n,k]_{q^m/q}$ sum-rank metric code (that is, a rank-metric code) with the property that $n>m$. 
Let $U$ be any system associated with $\C$ and consider $\mathcal{M}(U)$ the lift of $U$ and follow the proof of Theorem \ref{th:everysumrankisoneweight} to construct a one-weight sum-rank metric code $\C'$.
Since $U$ is the only $\fq$-subspace in $\mathcal{M}(U)$ having dimension larger than $m$, then in $\mathcal{M}(U)$ there cannot the an orbit of $U$ and hence $\C'$ cannot be obtained as (union) of orbital constructions.
\end{proof}

\begin{remark}
In the proof of the above theorem, we started from a rank-metric code, but then we can start from any code whose associated systems do not form an orbit under a transitive group of an $\fq$-subspace of dimension greater than $m$.
\end{remark}

\begin{remark}
In \cite{martinez2022doubly}, the author showed that in some cases extending an MSRD code by adding new blocks does not preserve the property of being MSRD.
\end{remark}

\medskip

\subsection{Constructions of two-weight minimal sum-rank metric codes}

In this section we will give examples of minimal sum-rank metric codes with two weights, in some cases by using the AB-condition and some other by exploiting the geometry behind them.
Let us start with the first construction.

\begin{construction}\label{con:disjscatt}
Consider $t$ mutually disjoint scattered $\mathbb{F}_q$-linear sets $L_{U_1},\ldots,L_{U_t}$ in $\PG(k-1,q^m)$ of rank $n_1, \ldots , n_t$, respectively, with $n_1\geq \ldots \geq n_t$. 
Suppose also that $L_{U_1}\cup \ldots\cup L_{U_t}\neq \PG(k-1,q^m)$.
Denote by $\C(U_1^{\perp'},\ldots,U_t^{\perp'})$ the geometric dual of a code associated with $(U_1,\ldots,U_t)$.
\end{construction}

The metric properties of the above construction are described in the following result.

\begin{theorem}\label{th:ranklists}
Let $\C$ be as in Construction \ref{con:disjscatt}. Its parameters are $[(mk-n_1,\ldots,mk-n_t),k]_{q^m/q}$ and it has two distinct nonzero weights:
\[ tm \text{ and } tm-1.\]
Moreover, its set of rank-lists is 
\[ S(\C)=\{ (m,\ldots,m),(m-1,m,\ldots,m),(m,m-1,\ldots,m),\ldots,(m,\ldots,m,m-1) \}. \]
\end{theorem}
\begin{proof}
Let start by observing that $(U_1^{\perp'},\ldots,U_t^{\perp'})$ is a system with parameters $[(mk-n_1,\ldots,mk-n_t),k]_{q^m/q}$. 
By Theorem \ref{th:pargeomdual}, the weights of $\C(U_1^{\perp'},\ldots,U_t^{\perp'})$ correspond to the possible values that the following expression can assume
\[ tm- \sum_{i=1}^t \dim_{\fq}(U_i^{\perp'}\cap \langle w \rangle_{\mathbb{F}_{q^m}}), \]
where $w \in \mathbb{F}_{q^m}^k\setminus\{0\}$.
Since the $L_{U_i}$'s are scattered and disjoint, we have that $\dim_{\fq}(U_i^{\perp'}\cap \langle w \rangle_{\mathbb{F}_{q^m}})\in \{0,1\}$ and can be one at most for one $i \in \{1,\ldots,t\}$ for every $w$, hence the weight distribution is determined.
The possible rank-lists of the codewords of $\C(U_1^{\perp'},\ldots,U_t^{\perp'})$ (see Remark \ref{rk:ranklistdual}) correspond to determine 
\[ (m-\dim_{\mathbb{F}_q}(U_1\cap \langle w \rangle_{\mathbb{F}_{q^m}}),\ldots, m-\dim_{\mathbb{F}_q}(U_t\cap \langle w \rangle_{\mathbb{F}_{q^m}})), \]
because of the assumptions on the $L_{U_i}$'s, we have that either $\dim_{\mathbb{F}_q}(U_1\cap \langle w \rangle_{\mathbb{F}_{q^m}})=0$ for any $i$ or there exist $j \in \{1,\ldots,t\}$ such that $\dim_{\mathbb{F}_q}(U_j\cap \langle w \rangle_{\mathbb{F}_{q^m}})=1$ and $\dim_{\mathbb{F}_q}(U_i\cap \langle w \rangle_{\mathbb{F}_{q^m}})=0$ for any the remaining values of $i$.
\end{proof}

The codes with the parameters as those in Construction \ref{con:disjscatt} all arise from Construction \ref{con:disjscatt}.

\begin{theorem}
If $\C$ is a sum-rank metric code with the parameters as in Construction \ref{con:disjscatt} then $\C$ can be obtain as in Construction \ref{con:disjscatt}.
\end{theorem}
\begin{proof}
Assume that $\C$ is an $[\mathbf{n},k]_{q^m/q}$ sum-rank metric code with $\textbf{n}=(mk-n_1,\ldots,mk-n_t)$ and with two distinct weights $tm$ and $tm-1$. Denote by $N=n_1+\ldots+n_t$ and let $(U_1,\ldots,U_t)$ any system associated with $\C$.
By Theorem \ref{th:connection}, we have that for any hyperplane of $\F_{q^m}^k$
\[ \sum_{i=1}^t \dim_{\fq}(U_i \cap H) \in \{ tm(k-1),tm(k-1)+1 \}, \]
and applying Proposition \ref{prop:weightdual} we obtain
\[ \sum_{i=1}^t \dim_{\fq}(U_i^{\perp'} \cap \langle w \rangle_{\F_{q^m}}) \in \{ 0,1 \}, \]
for any $1$-dimensional $\F_{q^m}$-subspace of $\F_{q^m}^k$.
Therefore, for any point $P \in \PG(k-1,q^m)$ we have that
\[ \sum_{i=1}^t w_{L_{U_i^{\perp'}}}(P)  \in \{ 0,1 \}, \]
and so $L_{U_i^{\perp'}}$'s are pairwise disjoint scattered $\F_q$-linear sets.
\end{proof}

Thanks to Theorem \ref{th:ranklists}, we can determine the weight distribution of the Hamming metric codes associated with those in Construction \ref{con:disjscatt}.

\begin{proposition}
An associated Hamming metric code to the codes of Construction \ref{con:disjscatt} has length $N=tkm-n_1\ldots-n_t$, dimension $k$ and exactly $|\{n_1,\ldots,n_t\}|+1$ nonzero distinct weights, which are
\[ \delta_{\max}=\sum_{i=1}^t\frac{q^{mk-n_i}-q^{mk-n_i-m}}{q-1}, \]
\[ w_j=\frac{q^{mk-n_j}-q^{mk-n_j-m+1}}{q-1} + \sum_{i\in\{1,\ldots,t\}\setminus\{j\}}\frac{q^{mk-n_i}-q^{mk-n_i-m}}{q-1},\,\,j \in \{1,\ldots,t\}, \]
and 
\[ \delta_{\min}=\frac{q^{mk-n_t}-q^{mk-n_t-m+1}}{q-1} + \sum_{i=1}^{t-1}\frac{q^{mk-n_i}-q^{mk-n_i-m}}{q-1}. \]
\end{proposition}
\begin{proof}
By taking into account the set of rank-lists of a code as in Construction \ref{con:disjscatt} determined in Theorem \ref{th:ranklists}, by Proposition \ref{prop:weight_G_Ext} the possible weights of an associated  code is given by the following formula
\[ \frac{q^{mk-n_j}-q^{mk-n_j-m+1}}{q-1} + \sum_{i\in\{1,\ldots,t\}\setminus\{j\}}\frac{q^{mk-n_i}-q^{mk-n_i-m}}{q-1}, \]
for any $j \in \{1,\ldots,t\}$ and by
\[\sum_{i=1}^t\frac{q^{mk-n_i}-q^{mk-n_i-m}}{q-1}.\]
Therefore, the number of distinct weights is given by the number of the different $n_i$'s plus one.
\end{proof}

\begin{remark}
In the case in which $n_1=\ldots=n_t=n$, such  codes are two-weight Hamming metric codes, with weights
\[ \delta_{\max}=t\cdot \frac{q^{mk-n}-q^{mk-n-m}}{q-1}, \]
and 
\[ \delta_{\min}=\frac{q^{mk-n}-q^{mk-n-m+1}}{q-1} + (t-1)\cdot \frac{q^{mk-n}-q^{mk-n-m}}{q-1}. \]
Note also that if the $n_i$'s are not all equal, then from a two-weight code in the sum-rank metric we obtain an Hamming metric code with more than two weights. This is a remarkable difference with the rank metric.
\end{remark}

\begin{theorem}\label{th:twoweightAB}
Let $\C$ be a code as in Construction \ref{con:disjscatt} with $n_1=\ldots=n_t=n$.
If 
\[t>(q-1)\cdot \frac{q^m}{q^m-1}, \]
then $\C$ is minimal  and it satisfies the AB-condition.
\end{theorem}
\begin{proof}
We will prove the minimality of $\C$ with the aid of the AB-condition.
To this aim
\[ \frac{\delta_{\max}}{\delta_{\min}}= \frac{t(q^m-1)}{q^m-q+(t-1)(q^m-1)}=1+\frac{q-1}{tq^m-q-t+1}, \]
since $t>(q-1)q^m/(q^m-1)$ we have 
\[ \frac{\delta_{\max}}{\delta_{\min}} < \frac{q^m}{q^m-1} \]
and hence Theorem \ref{th:ABcond} implies the assertion.
\end{proof}

\begin{remark}
Consider Construction \ref{con:disjscatt} by using more than $(q-1)\cdot \frac{q^m}{q^m-1}$ mutually disjoint subgeometries (which exists for instance when $\gcd(m,k)=1$, see Theorem \ref{th:partionsubg}). The above theorem implies that  Construction \ref{con:disjscatt} gives minimal codes.
\end{remark}

Note that in the above result if we consider $t=2$ and $q>2$, then the AB-condition is not satisfied (it is indeed satisfied for $q=2$). In the next result, we show that when $t=2$, even if the AB-condition is not satisfied, Construction \ref{con:disjscatt} still gives minimal codes and the two rank-metric codes defined by the two blocks are not minimal.

\begin{theorem}\label{th:twoweight2blocks}
Let $U_1$ and $U_2$ be two trivially intersecting scattered $\F_q$-subspaces of dimension $m$ contained respectively in $W_1$ and $W_2$, where $W_1$ and $W_2$ are two distinct $2$-dimensional $\mathbb{F}_{q^m}$-subspaces of $\F_{q^m}^3$. The geometric dual $\C(U_1^{\perp'},U_2^{\perp'})$ of a code associated with $(U_1,U_2)$ is a minimal sum-rank metric code, which satisfies the AB-condition if and only if $q=2$.
Moreover, the codes associated with $U_1^{\perp'}$ and $U_2^{\perp'}$, respectively, are not minimal.
\end{theorem}
\begin{proof}
Let start by computing the possible dimension of intersection between the $U_i^{\perp'}$'s and the $\F_{q^m}$-subspaces of $\F_{q^m}^3$ with dimension either one or two with the aid of Proposition \ref{prop:weightdual}:
\[ \dim_{\mathbb{F}_q}(U_i^{\perp'}\cap W^\perp)=\dim_{\mathbb{F}_q}(U_i\cap W)\in \{0,1,m\}, \]
for any $2$-dimensional $\mathbb{F}_{q^m}$-subspace $W$ and $\dim_{\mathbb{F}_q}(U_i^{\perp'}\cap W^\perp)=m$ if and only if $W=W_i$, for any $i \in \{1,2\}$, also
\[ \dim_{\mathbb{F}_q}(U_i^{\perp'}\cap \langle w\rangle_{\mathbb{F}_{q^m}}^\perp)=\dim_{\mathbb{F}_q}(U_i\cap \langle w\rangle_{\mathbb{F}_{q^m}})+m\in \{m,m+1\}, \]
for any $1$-dimensional $\mathbb{F}_{q^m}$-subspace $\langle w\rangle_{\mathbb{F}_{q^m}}$. 
In terms of linear sets, this means that 
\[ w_{L_{U_i^{\perp'}}}(P) \in \{0,1,m\}\,\,\text{and}\,\, w_{L_{U_i^{\perp'}}}(\ell) \in \{m,m+1\}, \]
for any point $P$ and any line $\ell$ of $\PG(2,q^m)$.
We now show that $(U_1^{\perp'},U_2^{\perp'})$ is a cutting system (which is equivalent to show that $\C(U_1^{\perp'},U_2^{\perp'})$ is minimal by Theorem \ref{th:correspondence}), which is equivalent to show that any line $\ell$ of $\PG(2,q^m)$ meets $L_{U_1^{\perp'}}\cup L_{U_2^{\perp'}}$ in at least two points.
Note that since $L_{U_1^{\perp'}}$ and $L_{U_2^{\perp'}}$ have rank $2m$, which is greater than $m+1$, then every line meet $L_{U_1^{\perp'}}$ and $L_{U_2^{\perp'}}$ in at least one point.
Denote by $P_1$ and $P_2$ the points defined by $W_1^{\perp'}$ and $W_2^{\perp'}$, respectively. 
Let $\ell$ be any line through $P_1$, since $P_1\ne P_2$ and $P_1 \notin L_{U_2^{\perp'}}$, then $\ell$ meets $L_{U_2^{\perp'}}$ in at least another point. Therefore, $|\ell \cap (L_{U_1^{\perp'}} \cup L_{U_2^{\perp'}})|\geq 2$. Similar arguments can be performed when considering lines through $P_2$, so assume that $\ell$ is a line not passing through neither $P_1$ nor to $P_2$. Since all the points different from $P_1$ and $P_2$ have weight either one or zero and the weight of $\ell$ is either $m$ or $m+1$, then $|\ell \cap L_{U_1^{\perp'}}|\geq 2$ and hence $(U_1^{\perp'},U_2^{\perp'})$ is a cutting system.
Finally, we show that $U_1^{\perp'}$ and $U_2^{\perp'}$ are not cutting. Indeed, by contradiction assume that $U_1^{\perp'}$ is a cutting system. Any line $\ell$ through $P_1$ has weight $m+1$, since $w_{L_{U_1^{\perp'}}}(P_1)=m$ and $L_{U_1^{\perp'}} \cap \ell$ has at least two points. This implies that all the lines through $P_1$ are contained in $L_{U_1^{\perp'}}$ and hence
\[ |L_{U_1^{\perp'}}| \geq (q^m+1)q^m+1, \]
which is a contradiction to the fact that $|L_{U_1^{\perp'}}| \leq \frac{q^{2m}-1}{q-1}$ by \eqref{eq:card}. 
Now, let 
\[ \delta_{\max}=\frac{2(q^{(k-1)m}-q^{(k-2)m})}{q-1} \]
and
\[ \delta_{\min}=\frac{2q^{(k-1)m}-q^{(k-2)m+1}-q^{(k-2)m}}{q-1}. \]
Then $\delta_{\max}/\delta_{\min}<q^m/(q^m-1)$ if and only if
\[ -\frac{-3q^m+2+q^{m+1}}{(-q^m+q-q^{m}+1)(q^m-1)} <0 \]
and hence if and only if
\[ -3q^m+2+q^{m+1}<0, \]
and, since it can be rewritten as $q^m(q-3)<-2$, this happens if and only if $q=2$. Therefore, the assumption of Theorem \ref{th:ABcond} are satisfied if and only if $q=2$.
\end{proof}

\begin{example}
Let 
\[ U_1=\{(x,x^q,0) \colon x \in \mathbb{F}_{q^m}\}\,\,\text{and}\,\, U_2=\{(0,x,x^q) \colon x \in \mathbb{F}_{q^m}\}.\]
It is easy to see that they satisfy the assumptions of Theorem \ref{th:twoweight2blocks}.
Consider $\sigma'$ as the following sesquilinear form
\[\sigma' \colon ((x,y,t),(x',y',z')) \in \mathbb{F}_{q^m}^3 \mapsto \mathrm{Tr}_{q^m/q}(xx'+yy'+zz') \in \F_q.\]
Then
\[ U_1^{\perp'}= \{ (x^{q^{m-1}},-x,y) \colon x,y \in \F_{q^m} \} \]
and 
\[ U_2^{\perp'}= \{ (y,x^{q^{m-1}},-x) \colon x,y \in \F_{q^m} \}. \]
Therefore, a generator matrix of $\mathcal{C}(U_1^{\perp'},U_2^{\perp'})$ is
\[ G=\left( 
\begin{array}{cccccc|cccccc}
a_1^{q^{m-1}} & \ldots & a_m^{q^{m-1}} & 0  & \ldots & 0 & a_1 & \ldots & a_m & 0 & \ldots & 0 \\
-a_1 & \ldots & -a_m & 0  & \ldots & 0 &  0 & \ldots & 0 & a_1^{q^{m-1}} & \ldots & a_m^{q^{m-1}} \\
0 & \ldots & 0 & a_1  & \ldots & a_m &  0 & \ldots & 0 & -a_1 & \ldots & -a_m
\end{array}
\right)\in \F_{q^m}^{3\times 4m},\]
where $\{a_1,\ldots,a_m\}$ is an $\fq$-basis of $\F_{q^m}$.
\end{example}

\begin{remark}
In the above example, we may replace $x^q$ with any scattered polynomial; see \cite{longobardi2021large,neri2021extending} and the references therein.
\end{remark}

In the next result we show that when $t=2$ in Construction \ref{con:disjscatt} with $n_1$ and $n_2$ having a distinct value, then the code obtained satisfies the AB-condition if $n_1$ and $n_2$ are enough close distinct numbers.

\begin{theorem}
Let $U_1$ and $U_2$ be two trivially intersecting scattered $\F_q$-subspaces of dimension $n_1$ and $n_2=n_1-r$ with $1\leq r \leq m$, respectively. The geometric dual $\C(U_1^{\perp'},U_2^{\perp'})$ of a code associated with $(U_1,U_2)$ is a minimal sum-rank metric code which satisfies the AB-condition.
\end{theorem}
\begin{proof}
As in the previous proof, we start by computing the possible dimension of intersection between the $U_i^{\perp'}$'s and the $\F_{q^m}$-subspaces of $\F_{q^m}^k$ with dimension $k-1$:
\[ \dim_{\mathbb{F}_q}(U_i^{\perp'}\cap \langle w\rangle_{\F_{q^m}}^\perp)=\dim_{\mathbb{F}_q}(U_i\cap \langle w\rangle_{\F_{q^m}})+(k-1)m-n_i\in \{(k-1)m-n_i, (k-1)m-n_i+1\}, \]
for any one-dimensional $\F_{q^m}$-subspace $\langle w\rangle_{\F_{q^m}}$ in $\F_{q^m}^k$.
Therefore, by Remark \ref{rk:ranklist} we have that the rank lists of the code $\C(U_1^{\perp'},U_2^{\perp'})$ are 
\[ (m,m), (m-1,m) \,\,\text{and}\,\, (m,m-1). \]
Considering $\C(U_1^{\perp'},U_2^{\perp'})^{\mathrm{H}}$ we have
\[ \delta_{\max}=\frac{q^{km-n_1}-q^{(k-1)m-n_1}+q^{km-n_2}-q^{(k-1)m-n_2}}{q-1} \]
and
\[ \delta_{\min}=\frac{q^{km-n_1}-q^{(k-1)m-n_1+1}+q^{km-n_2}-q^{(k-1)m-n_2}}{q-1}. \]
Replacing $n_1=n_2+r$, $\delta_{\max}/\delta_{\min}<\frac{q^m}{q^m-1}$ is equivalent to 
\[ -\frac{-2q^m+1-q^{r+m}+q^r+q^{m+1}}{(-q^m+q-q^{r+m}+q^r)(q^m-1)} <0 \]
and hence
\[ -2q^m+1-q^{r+m}+q^r+q^{m+1}<0, \]
which holds true as $m\geq r$ and $r\geq 1$. Therefore, the assertion follows by applying Theorem \ref{th:ABcond}.
\end{proof}

\bigskip

\section*{Acknowledgments}
The first author was partially supported by the ANR-21-CE39-0009 - BARRACUDA (French \emph{Agence Nationale de la Recherche}). The second author was supported by the project ``VALERE: VAnviteLli pEr la RicErca" of the University of Campania ``Luigi Vanvitelli'' and by the Italian National Group for Algebraic and Geometric Structures and their Applications (GNSAGA - INdAM). He is very grateful for the hospitality of the \emph{Universit\'e Paris 8}, France, where he was a visiting researcher for two weeks during the development of this research. 

\bigskip

\bibliographystyle{abbrv}
\bibliography{biblio}

\begin{thebibliography}{10}

\bibitem{Alfarano2020ageometric}
G.~N. Alfarano, M.~Borello, and A.~Neri.
\newblock A geometric characterization of minimal codes and their asymptotic
  performance.
\newblock {\em Advances in Mathematics of Communications}, 16(1):115, 2022.

\bibitem{alfarano2023outer}
G.~N. Alfarano, M.~Borello, and A.~Neri.
\newblock Outer strong blocking sets.
\newblock {\em arXiv preprint arXiv:2301.09590}, 2023.

\bibitem{alfarano2021linearcutting}
G.~N. Alfarano, M.~Borello, A.~Neri, and A.~Ravagnani.
\newblock Linear cutting blocking sets and minimal codes in the rank metric.
\newblock {\em Journal of Combinatorial Theory, Series A}, 192:105658, 2022.

\bibitem{alfarano2021three}
G.~N. Alfarano, M.~Borello, A.~Neri, and A.~Ravagnani.
\newblock Three combinatorial perspectives on minimal codes.
\newblock {\em SIAM Journal on Discrete Mathematics}, 36(1):461--489, 2022.

\bibitem{alfarano2021sum}
G.~N. Alfarano, F.~J. Lobillo, A.~Neri, and A.~Wachter-Zeh.
\newblock Sum-rank product codes and bounds on the minimum distance.
\newblock {\em Finite Fields and Their Applications}, 80:102013, 2022.

\bibitem{Expander}
N.~Alon, A.~Bishnoi, S.~Das, and A.~Neri.
\newblock Strong blocking sets and minimal codes from expander graphs.
\newblock {\em preprint}, 2023.

\bibitem{ashikhmin1998minimal}
A.~Ashikhmin and A.~Barg.
\newblock Minimal vectors in linear codes.
\newblock {\em IEEE Transactions on Information Theory}, 44(5):2010--2017,
  1998.

\bibitem{ball2000linear}
S.~Ball, A.~Blokhuis, and M.~Lavrauw.
\newblock Linear $(q+1)$-fold blocking sets in {PG}$(2,q^4)$.
\newblock {\em Finite Fields and Their Applications}, 6(4):294--301, 2000.

\bibitem{bartoli2023small}
D.~Bartoli and M.~Borello.
\newblock Small strong blocking sets by concatenation.
\newblock {\em SIAM Journal on Discrete Mathematics}, 37(1):65--82, 2023.

\bibitem{bartoli2018maximum}
D.~Bartoli, M.~Giulietti, G.~Marino, and O.~Polverino.
\newblock Maximum scattered linear sets and complete caps in {G}alois spaces.
\newblock {\em Combinatorica}, 38(2):255--278, 2018.

\bibitem{bartoli2022new}
D.~Bartoli, G.~Marino, and A.~Neri.
\newblock New {MRD} codes from linear cutting blocking sets.
\newblock {\em Annali di Matematica Pura ed Applicata (1923-)}, pages 1--28,
  2022.

\bibitem{bishnoi2023blocking}
A.~Bishnoi, J.~D'haeseleer, D.~Gijswijt, and A.~Potukuchi.
\newblock Blocking sets, minimal codes and trifferent codes.
\newblock {\em arXiv preprint arXiv:2301.09457}, 2023.

\bibitem{blokhuis2000scattered}
A.~Blokhuis and M.~Lavrauw.
\newblock Scattered spaces with respect to a spread in {PG}$(n,q)$.
\newblock {\em Geometriae Dedicata}, 81(1):231--243, 2000.

\bibitem{bonini2021minimal}
M.~Bonini and M.~Borello.
\newblock Minimal linear codes arising from blocking sets.
\newblock {\em Journal of Algebraic Combinatorics}, 53:327--341, 2021.

\bibitem{bonini2022saturating}
M.~Bonini, M.~Borello, and E.~Byrne.
\newblock Saturating systems and the rank covering radius.
\newblock {\em arXiv preprint arXiv:2206.14740}, 2022.

\bibitem{bonisoli1983every}
A.~Bonisoli.
\newblock Every equidistant linear code is a sequence of dual {H}amming codes.
\newblock {\em Ars Combinatoria}, 18:181--186, 1983.

\bibitem{byrne2021fundamental}
E.~Byrne, H.~Gluesing-Luerssen, and A.~Ravagnani.
\newblock Fundamental properties of sum-rank-metric codes.
\newblock {\em IEEE Transactions on Information Theory}, 67(10):6456--6475,
  2021.

\bibitem{calderbank1986geometry}
R.~Calderbank and W.~M. Kantor.
\newblock The geometry of two-weight codes.
\newblock {\em Bulletin of the London Mathematical Society}, 18(2):97--122,
  1986.

\bibitem{camps2022optimal}
E.~Camps-Moreno, E.~Gorla, C.~Landolina, E.~L. Garc{\'\i}a,
  U.~Mart{\'\i}nez-Pe{\~n}as, and F.~Salizzoni.
\newblock Optimal anticodes, {MSRD} codes, and generalized weights in the
  sum-rank metric.
\newblock {\em IEEE Transactions on Information Theory}, 68(6):3806--3822,
  2022.

\bibitem{csajbok2017maximum}
B.~Csajb{\'o}k, G.~Marino, O.~Polverino, and F.~Zullo.
\newblock Maximum scattered linear sets and {MRD}-codes.
\newblock {\em Journal of Algebraic Combinatorics}, 46(3):517--531, 2017.

\bibitem{davydov2011linear}
A.~Davydov, M.~Giulietti, S.~Marcugini, and F.~Pambianco.
\newblock Linear nonbinary covering codes and saturating sets in projective
  spaces.
\newblock {\em Advances in Mathematics of Communications}, 5(1):119--147, 2011.

\bibitem{el2003design}
H.~El~Gamal and A.~R. Hammons.
\newblock On the design of algebraic space-time codes for mimo block-fading
  channels.
\newblock {\em IEEE Transactions on Information Theory}, 49(1):151--163, 2003.

\bibitem{fancsali2014lines}
S.~Fancsali and P.~Sziklai.
\newblock Lines in higgledy-piggledy arrangement.
\newblock {\em the Electronic Journal of Combinatorics}, 21, 2014.

\bibitem{heger2021short}
T.~H{\'e}ger and Z.~L. Nagy.
\newblock Short minimal codes and covering codes via strong blocking sets in
  projective spaces.
\newblock {\em IEEE Transactions on Information Theory}, 68(2):881--890, 2021.

\bibitem{hirschfeld1998projective}
J.~Hirschfeld.
\newblock {\em Projective geometries over finite fields. Oxford mathematical
  monographs}.
\newblock Oxford University Press New York, 1998.

\bibitem{hu2022divisible}
J.~Hu, Q.~Liang, and R.~Calderbank.
\newblock Divisible codes for quantum computation.
\newblock {\em arXiv preprint arXiv:2204.13176}, 2022.

\bibitem{jurrius2017defining}
R.~Jurrius and R.~Pellikaan.
\newblock On defining generalized rank weights.
\newblock {\em Advances in Mathematics of Communications}, 11(1):225--235,
  2017.

\bibitem{lavrauw2016scattered}
M.~Lavrauw.
\newblock Scattered spaces in {G}alois geometry.
\newblock {\em Contemporary developments in finite fields and applications},
  pages 195--216, 2016.

\bibitem{lavrauw2015field}
M.~Lavrauw and G.~Van~de Voorde.
\newblock Field reduction and linear sets in finite geometry.
\newblock {\em Topics in finite fields}, 632:271--293, 2015.

\bibitem{longobardi2021large}
G.~Longobardi, G.~Marino, R.~Trombetti, and Y.~Zhou.
\newblock A large family of maximum scattered linear sets of
  {$\mathrm{PG}(1,q^n)$} and their associated {MRD} codes.
\newblock {\em arXiv:2102.08287}, 2021.

\bibitem{lu2005unified}
H.-f. Lu and P.~V. Kumar.
\newblock A unified construction of space-time codes with optimal
  rate-diversity tradeoff.
\newblock {\em IEEE Transactions on Information Theory}, 51(5):1709--1730,
  2005.

\bibitem{lunardon2008translation}
G.~Lunardon, G.~Marino, O.~Polverino, and R.~Trombetti.
\newblock Translation dual of a semifield.
\newblock {\em Journal of Combinatorial Theory, Series A}, 115(8):1321--1332,
  2008.

\bibitem{Martinez2018skew}
U.~Mart{\'\i}nez-Pe{\~n}as.
\newblock Skew and linearized {R}eed–{S}olomon codes and maximum sum rank
  distance codes over any division ring.
\newblock {\em Journal of Algebra}, 504:587--612, 2018.

\bibitem{martinez2019theory}
U.~Mart{\'\i}nez-Pe{\~n}as.
\newblock Theory of supports for linear codes endowed with the sum-rank metric.
\newblock {\em Designs, Codes and Cryptography}, 87(10):2295--2320, 2019.

\bibitem{martinezpenas2021hamming}
U.~Mart{\'\i}nez-Pe{\~n}as.
\newblock Hamming and simplex codes for the sum-rank metric.
\newblock {\em Designs, Codes and Cryptography}, 88(8):1521--1539, 2020.

\bibitem{martinez2022doubly}
U.~Mart{\'\i}nez-Pe{\~n}as.
\newblock Doubly and triply extended {MSRD} codes.
\newblock {\em arXiv preprint arXiv:2212.05528}, 2022.

\bibitem{martinez2022codes}
U.~Mart{\'\i}nez-Pe{\~n}as, M.~Shehadeh, F.~R. Kschischang, et~al.
\newblock Codes in the sum-rank metric: Fundamentals and applications.
\newblock {\em Foundations and Trends{\textregistered} in Communications and
  Information Theory}, 19(5):814--1031, 2022.

\bibitem{martinezpenas2018skew}
U.~Martínez-Peñas.
\newblock Skew and linearized {R}eed–{S}olomon codes and maximum sum rank
  distance codes over any division ring.
\newblock {\em Journal of Algebra}, 504:587--612, 2018.

\bibitem{massey1993minimal}
J.~L. Massey.
\newblock Minimal codewords and secret sharing.
\newblock In {\em Proceedings of the 6th joint Swedish-Russian international
  workshop on information theory}, pages 276--279, 1993.

\bibitem{neri2021extending}
A.~Neri, P.~Santonastaso, and F.~Zullo.
\newblock Extending two families of maximum rank distance codes.
\newblock {\em arXiv:2104.07602}, 2021.

\bibitem{neri2021thegeometryS}
A.~Neri, P.~Santonastaso, and F.~Zullo.
\newblock The geometry of one-weight codes in the sum-rank metric.
\newblock {\em Journal of Combinatorial Theory, Series A}, 194:105703, 2023.

\bibitem{nobrega2010multishot}
R.~W. N{\'o}brega and B.~F. Uch{\^o}a-Filho.
\newblock Multishot codes for network coding using rank-metric codes.
\newblock In {\em 2010 Third IEEE International Workshop on Wireless Network
  Coding}, pages 1--6. IEEE, 2010.

\bibitem{polverino2010linear}
O.~Polverino.
\newblock Linear sets in finite projective spaces.
\newblock {\em Discrete mathematics}, 310(22):3096--3107, 2010.

\bibitem{polverino2020connections}
O.~Polverino and F.~Zullo.
\newblock Connections between scattered linear sets and {MRD}-codes.
\newblock {\em Bulletin of the Institute of Combinatorics and its
  Applications}, 89:46--74, 2020.

\bibitem{Randrianarisoa2020ageometric}
T.~H. Randrianarisoa.
\newblock A geometric approach to rank metric codes and a classification of
  constant weight codes.
\newblock {\em Designs, Codes and Cryptography}, 88:1331–1348, 2020.

\bibitem{JohnPaolo}
P.~Santonastaso and J.~Sheekey.
\newblock {MSRD} codes and h-designs.
\newblock {\em in preparation}, 2023.

\bibitem{santonastaso2022subspace}
P.~Santonastaso and F.~Zullo.
\newblock On subspace designs.
\newblock {\em arXiv preprint arXiv:2204.13069}, 2022.

\bibitem{scotti}
M.~Scotti.
\newblock On the lower bound for the length of minimal codes.
\newblock {\em arXiv:2302.05350}, 2023.

\bibitem{tang2021full}
C.~Tang, Y.~Qiu, Q.~Liao, and Z.~Zhou.
\newblock Full characterization of minimal linear codes as cutting blocking
  sets.
\newblock {\em IEEE Transactions on Information Theory}, 67(6):3690--3700,
  2021.

\bibitem{vladut2007algebraic}
S.~Vladut, D.~Nogin, and M.~Tsfasman.
\newblock Algebraic geometric codes: basic notions, 2007.

\bibitem{zini2021scattered}
G.~Zini and F.~Zullo.
\newblock Scattered subspaces and related codes.
\newblock {\em Designs, Codes and Cryptography}, 89(8):1853--1873, 2021.

\end{thebibliography}

\end{document}